\documentclass{article}
\usepackage{zformat}
\numberwithin{equation}{section}

\title{Hikita conjecture for classical Lie algebras}
\begin{document}
\author{Do Kien Hoang} 
\date{}
\maketitle
\begin{abstract}
    Let $G$ be $Sp_{2n}$, $SO_{2n}$ or $SO_{2n+1}$ and let $G^\vee$ be its Langlands dual group. In \cite{Barbasch1985}, Barbasch and Vogan based on earlier work of Lusztig and Spaltenstein, define a duality map $D$ that sends nilpotent orbits $\bO_\ev \subset \fg^\vee$ to special nilpotent orbits $\bO_e\subset \fg$. In \cite{refinedbvls}, an upgraded version $\tilde{D}$ of this duality is considered, called the refined BVLS duality. $\tilde{D}(\bO_\ev)$ is a $G$-equivariant cover $\tilde{\bO}_e$ of $\bO_e$. Let $S_{\ev}$ be the nilpotent Slodowy slice of the orbit $\bO_{\ev}$. The two varieties $X^\vee= S_{\ev}$ and $X= \Spec(\bC[\tilde{\bO}_e])$ are expected to be symplectic dual to each other, \cite{refinedbvls}. In this context, a version of the Hikita conjecture predicts an isomorphism between the cohomology ring of the Springer fiber $\spr$ and the ring of regular functions on the scheme-theoretic fixed point $X^T$ for some torus $T$. This paper verifies the isomorphism for certain pairs $e$ and $e^\vee$. These cases are expected to cover almost all instances in which the Hikita conjecture holds when $e^\vee$ regular in a Levi $\fl^\vee\subset \fg^\vee$. Our results in these cases follow from the relations of three different types of objects: generalized coinvariant algebras, equivariant cohomology rings, and functions on scheme-theoretic intersections. We also give evidence for the Hikita conjecture when $e^\vee$ is distinguished.
\end{abstract}
\section{Introduction}
\subsection{Context}
Let $G$ be a semisimple Lie group, and let $G^\vee$ be its Langlands dual. Write $\fg$ and $\fg^\vee$ for their Lie algebras. In the following, by a \textit{nilpotent cover} we mean a $G$-equivariant cover of a nilpotent orbit $\bO\subset \fg$. In \cite[Section 9]{refinedbvls}, the authors introduce a refined BVLS (Barbasch-Vogan-Lusztig-Spaltenstein) duality map:
$$\tilde{D}: \{\text{nilpotent orbits of }\fg^\vee\} \hookrightarrow \{\text{equivalence classes of nilpotent covers of } \fg \}.$$

Consider a nilpotent element $e^\vee\in \fg^\vee$. Write $\bO_\ev$ for the orbit of $e^\vee$. The image $\tilde{D}(\bO_\ev)$ is a $G$-equivariant cover of $D(\bO_\ev)$ where $D$ is the BVLS duality defined by Barbasch and Vogan in \cite{Barbasch1985}. Let $e$ be an element in $D(\bO_\ev)$. Write $\bO_e$ for the adjoint orbit of $e$ and $\tilde{\bO}_e$ for the cover $\tilde{D}(\bO_\ev)$. When $\fg$ is of classical type, the nilpotent orbits in $\fg$ and $\fg^\vee$ are classified by partitions (\cite{collingwood1993nilpotent}). In type A, $\tilde{D}$ coincides with $D$; the partition of $e$ is obtained from the partition of $e^\vee$ by taking the transpose. In type B,C, or D, the combinatorial relation between the two partitions is more complicated; details are given in \cite[Theorem 5.2]{McGovern1994}. 

Let $T$ be a maximal torus of $G$ and let $\fh$ be its Lie algebra. The torus $T$ acts on $\tilde{\bO}_{e}$. Let $\nu: \bC^\times \rightarrow T$ be a generic cocharacter of $T$. In the notation of \cite[Sections 4,5]{Losev2017}, the \textit{Cartan subquotient} $\cC_\nu (\bC[\tilde{\bO}_{e}])$ is the algebra $$\bC[\tilde{\bO}_{e}]_0/\sum_{i>0} \bC[\tilde{\bO}_{e}]_{-i} \bC[\tilde{\bO}_{e}]_{i}$$ where the subscript denotes the grading induced by $\nu(\bC^\times)$. Alternatively, we can view $\cC_\nu (\bC[\tilde{\bO}_{e}])$ as the algebra of functions on the scheme-theoretic fixed point $\Spec(\bC[\tilde{\bO}_e])^{T}$. When $\tilde{\bO}_{e}= \bO_e$ and the closure of $\bO_e$ is normal, the algebra $\cC_\nu (\bC[\tilde{\bO}_{e}])$ is identified with $\bC[\overline{\bO}_e\cap \fh]$, the functions on the scheme-theoretic intersection $\overline{\bO}_e\cap \fh$. Consider the grading of this algebra so that the coordinate functions of $\fh$ have degrees $2$. 

Write $\spr$ for the Springer fiber over $e^\vee$. Write $S_\ev$ for the nilpotent Slodowy slice of $\ev$, a certain transverse slice of $\bO_\ev$ in the nilpotent cone of $\fg^\vee$ (see, e.g., \cite[Section 3.7]{CG}). When $\fg$ is of type $A$, a well-known result is that there is an isomorphism between the two graded algebras $H^*(\spr)$ and $\bC[\overline{\bO}_e\cap \fh]$. Various proofs and generalizations of this result can be found in \cite{deConcini1981}, \cite{Tanisaki}, \cite{Brundan2011},... The isomorphism $H^*(\spr) \cong \bC[\overline{\bO}_e\cap \fh]$ is often regarded as an instance of the Hikita conjecture for the pair $(S_\ev, \Spec(\bC[\bO_e]))$ in the context of symplectic duality, see, e.g., \cite{Hikita2016}. More generally, it is expected that the varieties $S_\ev$ and $\Spec(\bC[\tilde{\bO}_e]))$ are symplectic dual to each other; see \cite[Section 9.3]{refinedbvls}. Therefore, it makes sense to consider a version of the Hikita conjecture that predicts $H^*(\spr)\cong \cC_\nu(\bC[\tilde{\bO}_e])$. This is the subject of \cite{HMK2024}. In loc.cit., we consider a broader context in which the two varieties $S_\ev$ and the affininization of $\tilde{\bO}_e$ are replaced by suitable parabolic Slodowy slices.

In \cite[Section 6]{HMK2024}, the authors have pointed out that the expectation $H^*(\spr)\cong \cC_\nu(\bC[\tilde{\bO}_e])$ generally fails. We then propose various modifications to the Hikita conjecture and its deformed version. Here, the term "deformed Hikita conjecture" is from \cite[Section 9.3]{refinedbvls}. Despite these results, the question of when the original Hikita conjecture holds is still worth investigating. In particular, we want to determine all pairs $(e, e^\vee)$ so that $H^*(\spr)\cong \cC_\nu(\bC[\tilde{\bO}_e])$. The main goal of this paper is to provide an almost complete answer to this question in the setting that $\fg$ is of classical type, $e^\vee$ is regular in some Levi subalgebra $\lv\subset \fg^\vee$, the orbit $\bO_e$ has normal closure, and $\tilde{\bO}_{e}= \bO_e$. By “almost complete", we mean that we predict the precise cases where the conjecture holds and verify most of these cases. We also give evidence that the Hikita conjecture holds in some cases outside of this setting. It would be desirable to understand why the Hikita conjecture is true in these few cases. For example, in Section 5 we prove the Hikita conjecture for $e^\vee$ that has a 2-row partition. In these cases, the Springer fiber $\spr$ and the Slodowy variety admit realizations in terms of Nakajima's quiver varieties of type D (\cite{Henderson2014}).

\subsection{Main results}
Assume that $e^\vee$ is regular in a Levi subalgebra $\lv\subset \fg^\vee$. Write $\fl\subset \fg$ for the Langlands dual of $\lv$. From \cite[Section 3.2]{refinedbvls}, the orbit of $e$ in $\fg$ is obtained by the Lusztig-Spaltenstein induction from the orbit $\{0\}$ of $\fl$. Hence, we also write $\Ind_{\fl}^{\fg}(\{0\})$ for $\bO_e$. According to \cite[Sections 2,3]{refinedbvls}, there is a parabolic subalgebra $\fp\subset \fg$ with Levi $\fl$ so that the natural map $T^*(G/P)\rightarrow \fg$ has the image $\overline{\bO}_e$ and the preimage of $\bO$ is $\tilde{\bO}_e$. When discussing the Hikita conjecture in this case, we make the following assumptions.
\begin{enumerate}
    \item The closure of $\bO_e$ is normal.
    \item $\tilde{\bO}_{e}= \bO_e$. Equivalently, $e$ is birationally Richardson.
\end{enumerate}
Let $\ft_e\subset \fl\subset \fg$ be the center of $\fl$, then $\ft_e$ consists of semisimple elements. We say $x\in \ft_e$ is generic if the centralizer of $x$ in $\fg$ is precisely $\fl$. Consider a generic element $x$ in $\ft_e$. Let $J_{x}$ be the defining ideal of the closed orbit ${G.x}$ in $\fg$. Let $W$ be the Weyl group of $G$. The next paragraph follows from \cite[Proposition 4]{conjugacyclass}.

The natural grading of $\bC[\fg]$ induces a filtration on $J_{x}$. Then $J_{e}= \gr J_{x}$ is the defining ideal of $\overline{\bO}_e$ in $\fg$. Consider a Cartan subalgebra $\fh$ of $\fl$ so that $x\in \fh$. Write $I_{e}$ for the ideal of $\bC[\fh]$ such that $\bC[\overline{\bO}_e\cap \fh]= \bC[\fh]/I_e$. Then $I_e\subset \gr I_x$ where $I_{x}$ is the defining ideal of the orbit $W.x= G.x\cap \fh$ in $\fh$. Thus we get a surjective $W$-equivariant map of algebras
\begin{equation} \label{flat}
    \bC[\overline{\bO}_e\cap \fh] \twoheadrightarrow  \bC[\fh]/ \gr I_{x}.
\end{equation}
 Note that the ideal $\gr I_x$ does not depend on the choice of the generic element $x$. We have proved the following in \cite[Section 6.4]{HMK2024}.
 \begin{pro} \label{equivalent statements}
 Assume that $\overline{\bO}_e$ is normal and $e$ is birationally Richardson. The Hikita conjecture holds if and only if both of the following conditions hold.
 \begin{enumerate}
     \item The quotient morphism (\ref{flat}) is an isomorphism. We call it \textit{the weak flatness condition}.
     \item We have an isomorphism of rings $H^*(\spr)\cong \bC[\fh]/ \gr I_{x}$. 
 \end{enumerate}
\end{pro}

The central object of these two statements is the algebra $\bC[\fh]/ \gr I_{x}$. In Section 2, we attempt to describe this algebra as a quotient of $\bC[\fh]$. Our more general aim is to find a set of generators of $I_x$ that is \textit{uniform} in the following sense. Write $y_1,...,y_n$ for the coordinate functions of $\fh$ and $t_1,...,t_k$ for the coordinate functions of $\ft_e$. 
\begin{defin} \label{uniform}
    A set of polynomials $P_1,...,P_m$ in the variables $y_1,...,y_n, t_1,...,t_k$ is called a set of \textit{uniform generators} of $I_x$ if they satisfy the following two conditions.
\begin{enumerate}
    \item For each generic $x=(x_1,...,x_k)\in \ft_e$, the specializations of $P_1,...,P_m$ to the point $t_i=x_i$ generate the ideal $I_x$.
    \item The homogeneous ideal $\gr I_{x}$ is generated by the specializations of $P_1,...,P_m$ to the point $(t_1,...,t_k)= (0,...,0)$. 
\end{enumerate}
Consequently, the algebra $\bC[y_1,...,y_n,t_1,...,t_k]/(P_1,...,P_m)$ is flat over $\bC[t_1,...,t_k]$.
\end{defin}

Let $n$ be the rank of $\fg$. Let $\fg(a)$ be the rank $a$ Lie algebra of the same type as $\fg$. The Levi subalgebra $\fl\subset \fg$ takes the form $\fg(a)\times \prod_{i=1}^{k} \fgl_{b_i}$ for $a+b_1+...+b_k= n$, $b_1\geqslant ...\geqslant b_k$. In Section 2 we obtain a set of uniform generators of $I_x$ when $a=0$ and obtain a conjectural description of uniform generators for $a>0$. We check that our set of generators for $a\neq 0$ works in the case $b_1= 1,2$. Our results in this section can be considered as a generalization of the result in type A from \cite{GARSIA199282}.

In Section 3, we discuss when the weak flatness condition holds. When $\fg$ is of type C, the closure of $\bO_e$ is normal if and only if $2a\leqslant b_k$ (\cite[Section 10.1]{HMK2024}). Our main result in type C is as follows. 
\begin{thm}[\cref{main sp}]
    Consider $\fg= \fsp_{2n}$, and let $\fl= \fsp_{2a} \times \prod_{i=1}^{k} \fgl_{b_i}$ be a Levi subalgebra of $\fg$. The weak flatness condition holds for the orbit $\bO_e= \Ind_\fl^{\fg}(\{0\})$ if and only if $a=0$. Moreover, we have an explicit description of the algebra $\bC[\overline{\bO}_e\cap \fh]$ when $a=0$.
\end{thm}
One direction of this theorem has been stated in \cite[Theorem 4]{Tanisaki}. However, the proof in loc.cit. contains a mistake because the author's description of the ideal $\gr I_x$ is not correct; see \cref{improve Tanisaki}. With our description of $\gr I_x$ in Section 2, the proof of \cref{main sp} essentially relies on the existence of the \textit{symplectic Pfaffian} (\cref{symplectic Pfaffian}). This is our primary enhancement over the proof presented in \cite{Tanisaki}.

When $\fg$ is an orthogonal Lie algebra, we have a conjecture for the weak flatness condition. Recall that we write $\fl$ as $\fg(a)\times \prod_{i=1}^{k} \fgl_{b_i}$ with $b_1\geqslant...\geqslant b_k$. If $b_1=0$, we have $\fl=\fg$ and the orbit of $e$ is the zero orbit. Hence, we assume that $b_1\geqslant 1$. From \cite[Section 10.1]{HMK2024}, we have the following.
\begin{itemize}
    \item When $\fg= \fso_{2n+1}$, the orbit $\bO_e$ has normal closure if and only if $2a+1\geqslant b_1$.
    \item When $\fg= \fso_{2n}$, the orbit $\bO_e$ has normal closure if and only if $2a\geqslant b_1$ or $a=0$ and there is $1\leqslant l\leqslant k$ so that $b_1=...=b_l>b_{l+1}=...=b_k$. The latter condition is for the very even orbits in $\fso_{2n}$ to have normal closures (\cite{Kraft1982OnTG}).
\end{itemize}

\begin{conj}\label{weakflatso}\leavevmode
\begin{enumerate}
    \item Assume that $\fg$ is of type $B$. The weak flatness condition holds for the orbit $\bO_e= \Ind_\fl^{\fg}(\{0\})$ if and only if $a\geqslant b_1-1$.
    \item Assume that $\fg$ is of type $D$. The weak flatness condition holds for the orbit $\bO_e= \Ind_\fl^{\fg}(\{0\})$ if and only if at least one of the following conditions is satisfied.
    \begin{itemize}
        \item $a\geqslant b_1-1$,
        \item $k=1$, 
        \item or $b_1= 2$.
    \end{itemize}
    When $a=0$, the latter two conditions mean that the partition of $e$ has two rows or two columns.
\end{enumerate}
\end{conj}
In Section 3.3, we prove this conjecture for the case $b_1= 1,2$. One primary obstacle to proving this conjecture in more general cases is that our set of generators of $\gr I_x$ for the case $a\neq 0$ is conjectural. 

We now proceed to discuss the second condition in \cref{equivalent statements}. Because the geometry of $\spr$ may be complicated, it is challenging to check whether there is an isomorphism between the rings $H^*(\spr)$ and $\cO[\fh]/ \gr I_{x}$. Fortunately, there is a stronger condition that we can study. In particular, we have the following result.
\begin{pro}\label{graded isom}
    There exists an isomorphism of graded rings $H^*(\spr)\cong \cO[\fh]/ \gr I_{x}$ if and only if the pullback map $i^*: H^*(\cB^\vee)\rightarrow H^*(\spr)$ is surjective. Here, $\bv$ is the flag variety of $\fg^\vee$.
\end{pro}
\begin{proof}
    In \cite[Theorem 1]{carell} it has been shown that the isomorphism of graded rings $H^*(\spr)\cong \cO[\fh]/ \gr I_{x}$ would follow from the surjectivity of the pullback map $i^*: H^*(\cB^\vee)\rightarrow H^*(\spr)$. On the other hand, if $H^*(\spr)$ and $\cO[\fh]/ \gr I_{x}$ are isomorphic as graded rings, then $H^*(\spr)$ is generated in degree $2$. Let $\lambda^\vee$ be the partition of $e^\vee$. From \cite{Lehn_2011}, the map $i^*$ is surjective in degree $2$, except when $\lambda^\vee$ has two parts in type C or $\lambda^\vee= (2n-1,1,1)$ in type B. In these cases, it is not hard to check that $\dim H^2(\spr)$ is different from the dimension of the degree $2$ part of $\cO[\fh]/ \gr I_{x}$, see Section 5. Therefore, \cref{graded isom} follows. 

\end{proof}

There is numerical evidence that strongly suggests that we have $H^*(\spr)\cong \cO[\fh]/ \gr I_{x}$ as nongraded rings if and only if they are isomorphic as graded rings. For example, when $i^*$ is not surjective, we can check that the gradings of the two algebras are different using a computer computation. Another numerical value that can be compared is the dimensions of the socles, examples are given in \cite[Section 6]{HMK2024}. Therefore, in this paper, we focus on the surjectivity of $i^*$.

In \cite{HMK2024}, we have conjectured the following. Note that we do not assume that $e^\vee$ is regular in a Levi $\fl^\vee$ in this conjecture, although this happens to be a necessary condition in type B and type C.

\begin{conj}\cite[Conjecture B.7]{HMK2024} \label{surjectivity conj}
    The map $i^*$ is surjective in type B,C, and D in exactly the following cases.
\begin{enumerate}
    \item $\fg^\vee= \fsp_{2n}$, $e^\vee$ is regular in $\lv= \fsp_{2a}\times \prod_{i=1}^{k} \fgl_{b_i}$ for $b_i\in \{2a+1, 2a,2a-1, 1\}$, $1\leqslant i\leqslant k$.
    \item $\fg^\vee= \fso_{2n+1}$, $e^\vee$ is regular in $\lv= \fso_{2a+1}\times \prod_{i=1}^{k} \fgl_{b_i}$ for $b_i\in \{2a+2,2a+1,2a\}$, $1\leqslant i\leqslant k$.
    \item $\fg^\vee= \fso_{2n}$, and the partition $\lambda^\vee$ of $e^\vee$ takes one of the following forms
    \begin{itemize}
        \item $((2a+2)^{2d_1}, (2a+1)^{2d_2}),$ 
        \item $(2a+1)^{2d_1}, (2a)^{2d_2})$,
        \item $((2a+3)^{2d_1+1}, (2a+2)^{2d_2}, (2a+1)^{2d_3+1}),$
        \item $((2a+1)^{2d_1+1},(2b+1)^{2d_2+1}),$
        \item $((2a+1)^{2d_1+1},(2a)^{2d_2}, 2^{2d_3},1^{2d_4+1}),$
    \end{itemize}
    for some $a, b, d_1, d_2, d_3, d_4\geqslant 0$.
    \end{enumerate}
\end{conj}
In \cite[Appendix B]{HMK2024}, we have proved that for $i^*$ to be surjective, $e^\vee$ must satisfy the conditions in \cref{surjectivity conj}. In this paper, we introduce an approach to prove that these conditions are sufficient. In the cases listed above, we can find a large torus $T_\ev\subset G^\vee$ that acts on $\spr$. In particular,  if $e^\vee$ is regular in $\fl^\vee$, we choose $T_\ev$ to be the identity component of the centralizer of $\fl^\vee$ in $G^\vee$. We consider the pullback map in equivariant cohomology $i^{*}_{T_\ev}: H^*_{T_\ev}(\bv)\rightarrow H^*_{T_\ev}(\spr)$. By the graded version of Nakayma's lemma, $i^{*}_{T_\ev}$ is surjective if and only if $i^*$ is surjective. 

Write $\cA_{T_\ev}$ for the image of $i^{*}_{T_\ev}$. The algebra $\cA_{T_\ev}$ is closely related to the ideals $I_x$ mentioned in \cref{equivalent statements}. Specifically, if we have a set of uniform generators $P_1,...,P_m\in \bC[y_1,...,y_n,t_1,...,t_k]= \bC[\fh][\ft_{\ev}]$ for $I_x$, then $\cA_{T_\ev}$ is isomorphic to $\bC[\fh][\ft_\ev]/ (P_1,...,P_m)$. In this case, $\cA_{T_\ev}$ is a free module over $\bC[\ft_\ev]$ of rank $|W/W_L|= |W.x|$ where $W$ and $W_L$ are the Weyl groups of $\fg$ and $\fl$, respectively. On the other hand, the equivariant cohomology $H^*_{T_\ev}(\spr)$ is a free module over $\bC[\ft_\ev]$ of the same rank because the Springer fiber $\spr$ admits a $T_\ev$-equivariant affine paving (\cite{DLP}) and $\dim H^*(\spr)= |W/W_L|$ (\cite{induction}). View $i_{T_\ev}^*$ as a morphism between two locally free sheaves over $\ft_\ev$. By a version of the Hartogs Lemma, the isomorphism $\cA_{T_\ev}\cong H^*_{T_\ev}(\spr)$ would follow from the surjectivity of the restriction of the map $i^{*}_{T_\ev}$ to a certain open set $U\subset \ft_\ev$ such that $\ft_\ev\setminus U$ has codimension $2$. By choosing a suitable open set $U$, we can reduce the proof of \cref{surjectivity conj} to the case where the partition of $e^\vee$ has $2,3$, or $4$ parts. In Section 4 and Appendix A, we prove \cref{surjectivity conj} for certain cases of $e^\vee$. We expect this method to work in general, but we also emphasize that a crucial ingredient is the existence of a set of uniform generators for $I_x$ that has been conjectured in Section 2. 

In summary, modulo \cref{weakflatso} and \cref{surjectivity conj}, we have the following result.
\begin{conj}\label{main conj}
    The Hikita conjecture for $e^\vee$ regular in $\lv$ and $\tilde{\bO}_e= {\bO}_e$ with normal closure is true in exactly the following cases.
    \begin{enumerate}
    \item $\fg= \fsp_{2n}$, and $\fl$ takes the form $\prod_{i=1}^{l} \fgl_{2}\times \prod_{i=1}^{n-2l} \fgl_{1}$.
    \item $\fg= \fso_{2n+1}$, and $\fl$ takes the form $\fso_3\times \prod_{i=1}^{l} \fgl_2\times \prod_{j=1}^{n-2l-1} \fgl_1$, $\fso_5\times \prod \fgl_3\times \prod \fgl_1$, or $\fso_{2l+1}\times \prod_{j=1}^{n-l} \fgl_1$.
    \item $\fg= \fso_{2n}$, and $\fl$ is one of the following subalgebras $\fso_6\times \prod \fgl_4\times \prod \fgl_2\times \prod \fgl_1$, $\fso_4\times \prod_{i=1}^{l} \fgl_2\times \prod_{j=1}^{n-2l-2} \fgl_1$, $\fso_2\times \prod_{i=1}^{l} \fgl_2\times \prod_{j=1}^{n-l-1} \fgl_1$, $\fso_{2l}\times \prod_{j=1}^{n-l} \fgl_1$, $\fgl_n$, or $\prod_{i=1}^{n} \fgl_2$. The last two cases are only for $n$ even.
\end{enumerate}
\end{conj}

The main result of this paper is the verification of \cref{main conj} in the cases where $\fl$ does not contain any factor $\fgl_3$ or $\fgl_4$.
\begin{thm*}[\cref{surjective case}]
    In the cases listed in \cref{main conj}, if the $\fgl$ factors of $\fl$ have sizes up to $2$, the Hikita conjecture is true.
\end{thm*}

In Section 5, we prove the Hikita conjecture for the cases where the orbits $\bO_e$ are special spherical orbits in $\fso_{2n}$ and $\fso_{2n+1}$. In this setting, several conditions we considered previously are relaxed. In particular, these cases include the following features.
\begin{itemize}
    \item $e^\vee$ is distinguished.
    \item $\tilde{\bO}_e= \tilde{D}(\bO_{\ev})$ is a double cover of $\bO_e= D(\bO_{\ev})$.
    \item $\tilde{\bO}_e= \bO_e$, but the closure of $\bO_e$ is not normal.
    \item The orbit $\bO_e$ is very even in type D. This is the case $\fl= \fgl_n$ in \cref{main conj}.
\end{itemize}
A special case of our results in Section 5 is when $\bO_e$ is the minimal orbit in type D. The Hikita conjecture for this case has been studied in \cite{Shlykov2020}.

\subsection{Connection to \cite{HMK2024}}
Let $(\ev, \hv, f^\vee)$ be an $\fsl_2$-triple in $\fg^\vee$, and let $T_\hv\subset G^\vee$ be the torus that has $\text{Lie}(T_\hv)= \bC \hv$. In \cite{HMK2024}, we discuss the Hikita conjecture and its deformed version when $e^\vee$ is regular in a Levi subalgebra $\lv\subset \fg^\vee$. Specifically, the deformed Hikita conjecture predicts an isomorphism between the equivariant cohomology $H^*_{T_\ev}(\spr)$ (resp. $H^*_{T_\ev\times T_\hv}(\spr)$) and the functions on the scheme-theoretic fixed point of a base change of the universal Poisson deformation of $\tilde{\bO}_e$ (resp. a certain Cartan subquotient of the universal quantization deformation of $\tilde{\bO}_e$). As discussed in Section 1.2, these expectations do not always hold, and we need some modifications. 

In \cite{HMK2024}, we have changed our focus from $H^*_{T_\ev}(\spr)$ to the image of the pullback map $H^*_{T}(\bv)\rightarrow H^*_{T}(\spr)$ where $T= T_\ev$ or $T_\ev \times T_\hv$. Write $\cA_T$ for this image. This paper should be considered as a companion paper to \cite{HMK2024}, aiming to provide an explicit description of $\cA_T$. In particular, the results in Section 2 describe (conjecturally) $\cA_T$ for $T= T_\ev$. In addition, we study some cases for $T= T_\ev\times T_\hv$ in Appendix A. Our main tools to describe $\cA_T$ are from \cite[Section 7]{HMK2024}. In general, we expect $\cA_T$ to be a free module over $\bC[\ft]$ where $\ft$ is the Lie algebra of $T$. The algebra $\cA_{T_\hv}$ has also been studied for $e^\vee$ subregular in type D, see \cite[Theorem 1.3]{CXHY}. Note that their Hikita-Nakajima conjecture is our deformed Hikita conjecture. We expect to generalize their results to the cases in which the partitions of $e^\vee\in \fso_{2n}$ have two parts, using a quantized version of our arguments in Section 5.

The idea of studying the weak flatness condition and the surjectivity of $i^*: H^*(\bv)\rightarrow H^*(\spr)$ was proposed in \cite{HMK2024}. An application of the results in \cite[Section 8]{HMK2024} is that the deformed Hikita conjecture holds if these two conditions are satisfied. In \cite[Section 10.1]{HMK2024}, we explain that these conditions can only be fulfilled simultaneously in certain restrictive settings. These are precisely the cases listed in \cref{main conj}. In this paper, we study most of these cases and verify that the Hikita conjecture holds, hence implying the validity of the deformed version. Finally, Section 5 of this paper can be regarded as evidence for several conjectures proposed in \cite[Section 10.4]{HMK2024} in more general settings where $e^\vee$ may not be regular in a Levi.

\subsection{Acknowledgements} I would like to thank Vasily Krylov and Dmytro Matvieievskyi for various fruitful discussions related to the Hikita conjecture. I am grateful to Ivan Losev for countless valuable advice related to the project. Many results in this paper are first predicted using Sagemath, the software I learned from Elad Zelingher. This work is partially supported by the NSF under grant DMS-2001139.
\subsection{Notations and conventions} 
The following notations will be used consistently in the paper. We write $\fg^\vee$ and $\fl^\vee$ for the Langlands dual of $\fg$ and $\fl$. When we mention $\fl$ or $\fl^\vee$, the context is that $e^\vee$ is regular in $\fl^\vee$ and $\bO_e= \Ind_{\fl}^{\fg}(\{0\})$ is induced from the orbit $0$ of $\fl$. We use $W$ for the Weyl group of $\fso_{2n+1}$ (or $\fsp_{2n}$) and $W_D$ for the Weyl group of $\fso_{2n}$. The partition of $e^\vee$ is always denoted by $\lambda^\vee$. The meaning of $\lambda$ may vary in different sections; see below. 

We list some miscellaneous notation by sections. Some of them have already been introduced in previous sections. 

\textbf{Section 2}. 
\begin{itemize}
    \item $L$: subsets of $\{1,...,n\}$
    \item $e_p()$: the $p$-th elementary symmetric polynomials
    \item $e^{2}_{p}(L)$: $e_p(x_i^{2})_{i\in L}$ 
    \item $\lambda= (\lambda_0,...\lambda_m)$: the tranpose partition of $\lambda^\vee$
    \item $d_l(\lambda)$: $\lambda_{n-l}+...+ \lambda_m$
    \item $h^{2}_{p}(L)$: the $p$-th complete symmetric polynomials in variables $(x_i^{2})_{i\in L}$
    \item $y_L$: $\prod_{i\in L} y_i$
    \item $\ft_e$: $\fZ(\fl)$, the center of $\fl$
    \item $x$: a generic point of $\ft_e$ i.e, it satisfies $\fZ_\fg(x)= \fl$
    \item $\fh$: a Cartan subalgebra of $\fg$
    \item $I_x$: the defining ideal of $W.x\subset \fh$
\end{itemize}
\textbf{Section 3}
\begin{itemize}
    \item $\lambda$: the partition of $e$
    \item $\bO_\lambda$ or $\bO_e$: the orbit of $e$
    \item $I_\lambda$ or $I_e$: the defining ideal of the scheme-theoretic intersection $\overline{\bO}_e\cap \fh$ in $\fh$
    \item Pf, sPf: Pfaffian, symplectic Pfaffian
\end{itemize}
\textbf{Section 4}
\begin{itemize}
    \item $(\ev, \hv, f^\vee)$: an $\fsl_2$-triple in $\fg^\vee$
    \item $T, X$: a general torus and a general $T$-variety
    \item $\ft:$ the Lie algebra of $T$
    \item $X^t$: the subvariety of $X$ that is fixed under the action of the vector field induced by $t\in \ft$    
    \item $H^*_{T}(X)_t$: the specialization of $H^*_{T}(X)$ to $t\in \ft$
    \item $\bv$: the flag varieity of $\fg^\vee$
    \item $\spr$: the Springer fiber over $\ev$
    \item $i^*_{T}$: the pullback map $H_T^*(\bv)\rightarrow H_T^*(\spr)$
    \item $i^*_{\ev}$: $H^*(\bv)\rightarrow H^*(\spr)$, used when we consider the pullback maps for various different nilpotent elements
    \item $\fg^\vee(a)$: the Lie algebra of same type as $\fg^\vee$ with root system of size $a$    
\end{itemize}

\section{Defining ideal of the orbit $W.x$ in $\fh$}
Consider $\fg= \fso_{2n+1}$, $\fl= \fso_{2a+1}\times \prod_{i=1}^k \fgl_{b_i}$ for $b_1\geqslant b_2\geqslant... \geqslant b_k$ and $a\geqslant 0$. We assume that $\fl= \diag(\fgl_{b_1}, -\fgl_{b_1},...,\fgl_{b_k}, -\fgl_{b_k}, \fso_{2a+1})$ in $\fg$. Then $\ft_e$, the center of $\fl$, consists of elements of the form $x= \diag(t_1.\Id_{b_1},-t_1.\Id_{b_1},...,t_k.\Id_{b_k}, -t_k.\Id_{b_k},0.\Id_{2a+1})$. Pick a generic element $x\in \ft_e$. In terms of coordinates, genericity means $t_1,...,t_k$ nonzero and pairwise distinct. 

Recall that we write $I_{x}$ for the defining ideal of $W.x$ in $\fh$. In this subsection, we will use some properties of the finite set $W.x$ to obtain sets of (conjectural) generators of $I_{x}$ and $\gr I_{x}$. Our method follows the ideas of \cite{GARSIA199282} and \cite{Tanisaki}.
\begin{rem}\label{all types}
    The ideal $I_x$ depends only on the finite set $W.x$ and $\fh$. Because $\fso_{2n+1}$ and $\fsp_{2n}$ have the same Weyl group, the results in this section apply to $\fg= \fsp_{2n}$ and $\fl= \fsp_{2a}\times \prod_{i=1}^k \fgl_{b_i}$. 
    
    Let $W_D\subset W$ be the Wely group of $\fso_{2n}$. In the following, when $\fg= \fso_{2n}$, we only consider the case $a\geqslant 1$. In this case, the two sets $W_D.x$ and $W.x$ are the same. Therefore, in this section, the type of $\fg$ is not really concerned. When $a=0$ and $\fg= \fso_{2n}$, the element $e\in \fso_{2n}$ is very even. We discuss one of such cases in Section 5.
\end{rem}
\subsection{Case $a= 0$}
We first treat the case $a=0, b_1+...+ b_k= n$. Let $y$ be an arbitrary element of the set $W.x$. We view $y$ as a point of $\fh= \bC^n$, with $n$ coordinate functions $y_1,...,y_n$. For a subset $L\subset \{1,2,...,n\}$ of cardinality $1\leqslant l\leqslant n$, consider the polynomial $P_L(z,y)= \prod_{i\in L}(z^2- y_i^{2})$. Assume that $l> n-b_i$ for some $1\leqslant i\leqslant k$, then in $l$ arbitrary coordinates of $y\in W.x$, at least $l- (n-b_i)$ of them square to $t_i^{2}$. In other words, the polynomial $P_L(z,y)$ is divisible by 
$$Q_l(z)= \prod_{i=1}^{k} (z^2-t_i^{2})^{max(0, l- (n-b_i))}$$
for $y\in W.x$. 

Let $d_l$ be the degree of $Q_l(z)$. Write $R(z)$ for the remainder of the long division $P_L(z,y)$ by $Q_l(z)$ considered as two polynomials of $z$. From \cite[Proposition 3.1]{GARSIA199282}, if we view $t_1,...,t_k$ as variables, the coefficients of $R_{L}(z)$ are homogeneous polynomials in $y_1,...,y_n, t_1,...,t_k$. If we fix $x$, then for $0\leqslant d\leqslant d_l-1$, the coefficients $R_d$ of $z^d$ in $R(z)$, considered as elements of $\bC[y_1,...,y_n]$, belong to $I_x$. Write $e^{2}_{p}(L)$ for the $p$-th elementary symmetric polynomial in $l$ variables $y_{i}^{2}$, $i\in L$. From the proof of \cite[Proposition 3.1]{GARSIA199282}, the leading components of $R_d\in \bC[y_1,...,y_n]$ are precisely the coefficients of $z^d$ in $P_L(z,y)$ for $d\leqslant d_l-1$. Therefore, the polynomials $e^{2}_{p}(L)$ are elements of $\gr I_x$ for $2p \geqslant 2l-d_l+ 1$. The main result of this section is that these polynomials generate $\gr I_x$. Consequently, the coefficients of $R_{L}(z)$ give a set of uniform generators for $I_x$.
\begin{exa} \label{just sym}
    Consider $L=\{1,...,n\}$. Write $e_i$ for the $i$-th elementary symmetric polynomial. The coefficients of $R(z)$ (up to signs) are $e_{i}(y_1^{2},...,y_n^{2})- e_i(t_1^{2}, t_1^{2},..., t_{k}^{2})$, here $t_j$ appears $b_j$ times. The corresponding leading components are $e_i(y_1^{2},...,y_n^{2})$, $1\leqslant i\leqslant n$, they clearly belong to $\gr I_x$. 
    
    Assume that all $b_i$ are $1$, or equivalently, the stabilizer of $x$ is trivial in $W$. We have $d_l= 0$ for $l< n$, so the set of generators of $\gr I_x$ we obtain by this method consists of precisely $e_i(y_1^{2},...,y_n^{2})$ for $1\leqslant i\leqslant n$. 
\end{exa}

Before stating the result, we rephrase the statements in terms of $\lambda$, the dual partition of $(b_1,b_1,...,b_k,b_k)$. Writing it in standard form $(\lambda_0,...,\lambda_{m})$, we have $\lambda_0\geqslant...\geqslant \lambda_m$ and they are even. Note that the degrees of $Q_l$ are precisely $d_l= d_l(\lambda)= \lambda_{n-l}+... \lambda_{m}$. Later we will need to consider various partitions $\lambda$, so we write $d_l(\lambda)$ as a function of $\lambda$. Write $T_\lambda$ for the ideal of $\bC[y_1,...,y_n]$ geneated by the polynomials $e^{2}_{p}(L)$ for $L\subset \{1,...,n\}$ and $2p\geqslant 2l-d_l(\lambda)+1$. Note that $d_l(\lambda)$ is always even, so the condition on $p$ is equivalent to $p> l- d_l(\lambda)/2$. Write $R_\lambda$ for the quotient ring $\bC[y_1,...,y_n]/T_\lambda$. 

\begin{pro} \label{square Tanisaki}
    We have 
    \begin{equation} \label{main TS}
        \dim_{\bC} R_{\lambda}= \frac{n!\times 2^{n}}{b_1!...b_k!}.
    \end{equation}
    As a consequence, $\gr I_x$ coincides with $T_\lambda$ and is generated by $e^{2}_{p}(L)$ for $2p \geqslant 2l-d_l(\lambda)+ 1$ where $l$ is the cardinality of $L$.
\end{pro}
\begin{proof}
    We have a surjection $R_{\lambda}\twoheadrightarrow \bC[y_1,...,y_n]/\gr I_x$. Because the ideal $I_x$ defines a set of $\frac{n!\times 2^{n}}{b_1!...b_k!}$ points in $\bC^{n}$, the quotient $\bC[y_1,...,y_n]/I_x$ has dimension $\frac{n!\times 2^{n}}{b_1!...b_k!}$ over $\bC$, and therefore the same is true for $\bC[y_1,...,y_n]/\gr I_x$. Thus, the description of $\gr I_x$ follows from the equality (\ref{main TS}). We now prove (\ref{main TS}), it is clear that we only need to show $\dim_{\bC} R_{\lambda}\leqslant \frac{n!\times 2^{n}}{b_1!...b_k!}$. In the rest of the proof, we simply write $\dim $ for $\dim_\bC$.

    We prove (\ref{main TS}) by induction on $n$, the statement is true for $n=1$ (\cref{just sym}). For $0\leqslant j\leqslant m$, write $\lambda^{j}$ for the partition that we obtain from $\lambda$ by replacing $\lambda_j$ by $\lambda_j-2$. From the induction hypothesis, the dimensions of $R_{\lambda^{j}}$ are $\frac{b_i\times (n-1)!\times 2^{n}}{b_1!...b_k!}$ for the index $i$ such that $b_{i-1}\leqslant j< b_{i}$. The number $\frac{n!\times 2^{n}}{b_1!...b_k!}$ is precisely $\sum_{j=0}^{m} \dim R_{\lambda^j}$, so our aim is to prove that 
    \begin{equation}\label{dim inq}
        \dim R_\lambda\leqslant \sum_{j=0}^{m} \dim R_{\lambda^{j}}.
    \end{equation}

    The ring $R_\lambda$ has a filtration by the submodules $F_r= y_n^{r}R_\lambda$. We want to show that each quotient $F_r/F_{r+1}$ has the structure of a cyclic $R_{\lambda^{j}}$-module where $j$ is the index such that $\lambda_j> r\geqslant \lambda_{j+1}$. Then (\ref{dim inq}) follows as a corollary. More precisely, we prove that the map $$f_j: R_{\lambda^{j}}\rightarrow F_r/F_{r+1}, \,  t \mapsto ty_n^{r}$$ is a well-defined ring homomorphism. The claim about $R_{\lambda^{j}}$-module structure will naturally follow. 

    For $f_j$ to be well-defined, we need to show that for certain $p$, the elements of the forms $y_n^{r}e^{2}_{p}(L)$, $L\subset \{1,...,n-1\}$ are in the ideal $(y_n^{r+1})$ of $R_\lambda$. For notational simplicity, we assume $L= \{1,...,l\}$. Consider the polynomial $\prod_{i=1}^{l} (z^2-y_i^{2})$, write it as $\sum_{i=0}^{2l} c_{d} z^{d}$. Our task is to prove $y_n^{r}c_{d}$ belongs to the ideal $(y_n^{r+1})$ of $R_\lambda$ for $d\leqslant d_l(\lambda^{j})-2$. In the following, we only consider $l$ so that $d_l(\lambda^{j})>0$.
    
    By the definition of $T_\lambda$, it contains the coefficient of $z^{d}$ for $d\leqslant d_{l+1}(\lambda)- 1$ of the polynomials $(t^2-y_n^{2})\prod_{i=1}^{l} (z^2-y_i^{2})$. Hence, we have $c_{d-2}= y_{n}^{2}c_{d}$ in $R_\lambda$ for $d\leqslant d_{l+1}(\lambda)$. We consider two cases.
    \begin{enumerate}
        \item $d_{l+1}(\lambda)= d_l(\lambda^j)+2$. Then for $d\leqslant d_l(\lambda^j)$, we have $c_d= y_n^{2}c_{d+2}$. Consequently, $y_n^rc_d\subset (y_n^{r+1})$ in $R_\lambda$.
        \item $d_{l+1}(\lambda)= d_l(\lambda^j)$, equivalently $\lambda_{n-l-1}< \lambda_j$. By the same argument as in the previous case, we only need to show $y_n^rc_{d_l(\lambda^j)-2}\subset (y_n^{r+1})$. For $d\leqslant d_{l+1}(\lambda)= d_l(\lambda^{j})$, we have $c_{d-2}= y_{n}^{2}c_{d}$. Next, $c_d= 0$ in $R_\lambda$ for $d\leqslant d_l(\lambda)-2$. By definition, $d_{l+1}(\lambda)-d_{l}(\lambda)= \lambda_{n-l-1}\leqslant \lambda_{j+1}\leqslant r$. If $d_l(\lambda)\geqslant 2$, we have in $R_\lambda$, 
        $$y_{n}^{r}c_{d_l(\lambda^j)-2}= y_{n}^{r}c_{d_{l+1}(\lambda^j)-2}= y_{n}^{r-2}c_{d_{l+1}(\lambda^j)-3}=...= y_{n}^{r-\lambda_{n-l-1}} c_{d_l(\lambda)-2}= 0.$$
        If $d_l(\lambda)= 0$, then write $y_{n}^{r}c_{d_l(\lambda^j)-2}$ as $y_{n}^{r-\lambda_{n-l-1}+2}c_{0}$. This element is divisible by $y_n^{2}c_0$, the free coefficient of $(z^2-y_n^{2})\prod_{i=1}^{l} (z^2-y_i^{2})$. This free coefficient is $0$ in $R_\lambda$ because we are considering $d_{l+1}(\lambda)= d_l(\lambda^j)>0$. Therefore, we always have $y_{n}^{r}c_{d_l(\lambda^j)-2}= 0$.
    \end{enumerate}
\end{proof}
Recall from \cref{uniform} that we have the notion of uniform generators for $I_x$. A corollary of \cref{square Tanisaki} is as follows.

\begin{cor}\label{ideal of Wx}
    Recall that we have the remainders $R_{L}(z)$ of the long division $P_L$ by $Q_l$. The coefficients of $R_{L}(z)$ are homogeneous polynomials in $y_1,...,y_n, t_1,...,t_k$. \cref{square Tanisaki} shows that the coefficients of $z^d$ in $R_{L}(z)$ for $d<d_l$ generate $I_x$ and the leading terms of these coefficients generate $\gr I_x$. Therefore, these coefficients form a set of uniform generators for $I_x$.
\end{cor}

On the one hand, our construction of $T_\lambda$ and the proof of \cref{square Tanisaki} closely resemble the results of type A from \cite[Section 3]{GARSIA199282} and \cite[Theorem 1]{Tanisaki}. On the other hand, we emphasize that the ideal $T_\lambda$ is original in the sense that it cannot be directly obtained from the results in type A. In the following, we explain the details.  

In \cite{Tanisaki}, the author instead considers $W.x$ as a subset of $\bC^{2n}$. Let the corresponding coordinate functions be $(y_1,...,y_{2n})$. A similar construction was carried out as follows. Consider $L'\subset \{1,2,...,2n\}$ with the cardinality $l'$. We can use the fact that $\prod_{i\in L'} (z- y_i)$ is divisible by $\prod_{i=1}^{k} (z^{2}- t_i^{2})^{max(0,l'-(2n-b_i)}$ to obtain an ideal $T_{\lambda}'\subset \bC[y_1,...,y_{2n}]$. Consider the quotient map $q_{\fh}: \bC[y_1,...,y_{2n}] \xrightarrow{/(y_i+y_{2n+1-i})} \bC[y_1,...,y_n]$. Then $q_{\fh}(T_\lambda'+ (y_i+y_{2n+1-i}))$ is a subset of $\gr I_x= T_\lambda$. The following example demonstrates a case where the inclusion $q_\fh(T'_\lambda+ (y_i+y_{2n+1-i}))\subset T_\lambda$ is a strict inclusion.
\begin{exa}\label{improve Tanisaki}
    Consider $b_1=3, b_2=1$, so $\lambda= (4,2,2)$ and $n= 4$. The point $x$ is $(t_1,t_1,t_1,t_2)$. Our ideal $T_\lambda$ is generated by $\sum_{i=1}^{4} x_i^{2}$ and $x_i^{2}x_j^{2}$ for $i\neq j$. The elements $x_i^{2}x_j^{2}$ are not in $q_{\fh}(T_\lambda'+ (y_i+y_{2n+1-i}))$. More precisely, we only have $x_i^{2}x_j^{2}+ x_j^{2}x_k^{2}+ x_k^{2}x_i^{2}$ in $q_{\fh}(T_\lambda'+ (y_i+y_{2n+1-i}))$ for $i,j,k$ pairwise distinct. This phenomenon implies the geometric fact that the intersection of the closure of a nilpotent orbit in $\fgl_{2n}$ and $\fsp_{2n}\subset \fgl_{2n}$ may not be reduced. 
\end{exa}

To conclude this section, we recall a technical result on symmetric polynomials. For a nonempty set $L\subsetneq \{1,...,n\}$, write $L^c$ for the complement $\{1,...,n\}\setminus L$. Recall that we write $e^{2}_{p}(L)$ for the $p$-th elementary symmetric polynomials of $l$ variables $y_i^{2}$, $i\in L$. Write $h^{2}_{p}(L^{c})$ for the $p$-th complete symmetric polynomials of $n-l$ variables $y_i^{2}$, $i\in L^{c}$. Let $I_n$ be the ideal of $\bC[y_1,...,y_n]$ that is generated by the symmetric polynomials of $y_1^{2},..., y_n^{2}$.
\begin{lem}\cite[Proposition 2.2]{Biagioli2008} \label{e2p}
    For $1\leqslant p\leqslant l$, we have $e^{2}_{p}(L)- (-1)^{p}h^{2}_{p}(L^c)\in I_n$. 
\end{lem}
\begin{cor} \label{power to eh}
    As a consequence, in the notation of \cref{square Tanisaki}, we have $y_i^{2k}\in T_\lambda$ because $e^{2}_{k}(L)\in T_\lambda$ for $|L|= n-1$.
\end{cor}
\subsection{Case $a\neq 0$}
Now the point $x$ takes the form $(t_1,...,t_1,....,t_k,...,t_k,0,...,0)$ where $0$ appears $a= n-\sum_{i=1}^{k} b_i$ times. For $1\leqslant l\leqslant n$, write $d'_l$ for the number $\sum_{i=1}^{k}max(0, l-(n-b_i))+ max(0, l- (n-a))$. Similarly to the previous section, for each $L\subset \{1,...,n\}$, we have $e^{2}_{p}(L)\in \gr I_x$ for $2p\geqslant 2l- d_l'+ 1$ where $l= |L|$. Next, in any $n-a+1$ arbitrary coordinates of $y\in W.x$, there must be at least one $0$. Hence, the elements $y_K:= \prod_{i\in K} y_i$ belong to $I_x$ and $\gr I_x$ for $K\subset \{1,...,n\}$, $|K|= n-a+1$. 

Write $\lambda(b)$ for the dual partition of $(b_1,b_1,...,b_k,b_k)$, it is a partition of $2n-2a$. Write $x'$ for the point $(t_1,...,t_1,...,t_k,...,t_k)\in \bC^{n-a}$ and write $W'$ for the Weyl group of signed permutations on these coordinates. Consider $l$ coordinates $y_{i_1},...,y_{i_l}$ of $y\in W.x$ so that $1\leqslant i_1,...,i_l\leqslant b_1+....+b_k$. Provided these $l$ coordinates are not zero, then they are governed by the conditions that come from $I_{x'}= T_{\lambda(b)}$ in the notation of \cref{square Tanisaki}. This observation gives rise to a set of elements in $I_x$ as follows. 

Consider $I_{x'}$ as a subset of $\bC[y_1,...,y_{n-a}]\subset \bC[y_1,...,y_n]$. Consider a polynomial $f\in I_{x'}$ and assume that all the variables that appear in $f$ are $y_{i_1},...,y_{i_l}$. Then the product $(y_{i_1}...y_{i_l})f$ is an element of $I_x$. Similarly, the same process works if we replace $I_{x'}$ and $I_x$ by $\gr I_x$ and $\gr I_{x'}$. Consequently, we have the following proposition. 

\begin{pro} \label{vague generators} \leavevmode
Let $\sco: \bC[y_1,...,y_{n-a}]\rightarrow \bC[y_1,...,y_{n}]$ be the map that sends $f$ to the product of $f$ and all the variables $y_i$ that appear in $f$. Let $T(b)$ be the following set of partial elementary and partial complete symmetric polynomials  
    $$\{e^{2}_{p}(L), h^{2}_{p}(\{1,...n-a\}\setminus L), L\subset\{1,...,n-a\}, l= |L|, 2p> 2l- d_l(\lambda(b))\}.$$
\begin{enumerate}
    \item We then have $\sco(I_{x'})\subset I_x$ and $\sco(\gr I_{x'})\subset \gr I_x$. 
    \item From \cref{square Tanisaki} and \cref{e2p}, $T(b)\subset \gr I_{x'}$. Hence, $\sco(T(b))\subset \gr I_x$.
\end{enumerate}
\end{pro}

The following example demonstrates how $\sco$ works and why we consider $T(b)$ instead of the set of generators of $T_{\lambda(b)}$ in \cref{square Tanisaki}. An important aspect is that the ideal generated by $\sco(e^{2}_{p}(L))$ for $p, L$ as described in \cref{vague generators} might be a proper subset of the ideal generated by $\sco(T(b))$.
\begin{exa}
    Consider $b_1=3, b_2= b_3 = 1$, and $a= 1$. We have $n= 6$, $\lambda(b)= (6,2,2)$. A set of generators of $T_{\lambda(b)}$ includes $e_p(y_1^{2},...,y_5^{2})$ for $p\leqslant 3$ and $y_i^{2}y_j^{2}y_k^{2}$ for $1\leqslant i<j<k\leqslant 5$. Here we skip the elements $e_{3}^{2}(L)$ with $|L|=4$ because they are generated by $y_i^{2}y_j^{2}y_k^{2}$. We then have $\sco(e_p(y_1^{2},...,y_5^{2}))= y_1...y_5e_p(y_1^{2},...,y_5^{2})$ and $\sco(y_i^{2}y_j^{2}y_k^{2})= y_i^{3}y_j^{3}y_k^{3}$. Next, we point out a new generator of $\gr I_x$ originating from $\sco(T(b))$. 

    Because $e_{3}^{2}(L)\in T_{\lambda(b)}$ for $|L|=3,4$, we have $p_{3}^{2}(L^c)\in T_{\lambda(b)}$ for $|L^c|= 1,2$ (\cref{e2p}). Hence, $T(b)$ contains $x_1^{6}, x_2^{6}, x_1^{6}+ x_1^{4}x_2^{2}+ x_1^{2}x_2^{4}+ x_2^{6}$. By the definition of $\sco$, we have $\sco(x_i^{6})= x_i^{7}$ for $i=1,2$, and $\sco (x_1^{6}+ x_1^{4}x_2^{2}+ x_1^{2}x_2^{4}+ x_2^{6})= x_1x_2(x_1^{6}+ x_1^{4}x_2^{2}+ x_1^{2}x_2^{4}+ x_2^{6})$. Consequently, the ideal generated by $\sco(T(b))$ contains $x_1^{5}x_2^{3}+x_1^{3}x_2^{5}$. This element is not generated by the elements in the previous paragraph. In fact, $\sco(f)$ is not generated by the elements $e^{2}_{p}(L)$ and $y_K$ that we have mentioned before \cref{vague generators}. Therefore, $\sco(T(b))$ gives rise to new generators of $\gr I_x$.

\end{exa}
We propose the following conjecture.
\begin{conj} \label{conj a not 0}
    When $a\neq 0$, the ideal $\gr I_x$ is generated by the following elements.
    \begin{enumerate}
        \item $e^{2}_{p}(L)$ for $L\subset \{1,...,n\}$, $|L|= l$, and $2p\geqslant 2l- d_l'+ 1$. Here $d'_l= \sum_{i=1}^{k}max(0, l-(n-b_i))+ max(0, l- (n-a))$.
        \item $y_K= \prod_{i\in K} y_i$ for $K\subset \{1,...,n\}$, $|K|= n-a+1$.
        \item $W.\sco(T(b))$, the $W$-stable set that is generated by $\sco(T(b))$.
    \end{enumerate}
\end{conj}
\begin{rem}
    Similarly to \cref{conj a not 0}, we have a corresponding conjectural set of generators for the ideal $I_x$. In particular, we still have $y_K\in I_x$, we replace $e^{2}_{p}(L)$ with certain remainders of polynomial divisions as in \cref{ideal of Wx}, and we replace $W.\sco(T(b))$ with $W.\sco(T'(x'))$ for a set of uniform generators $T'(x')$ of $I_{x'}$. Here, the existence of $T'(x')$ is explained in \cref{ideal of Wx}. Therefore, this set of generators is uniform in the sense of \cref{uniform} if \cref{conj a not 0} is true.
\end{rem}
In practice, once the numbers $a$ and $b_i$ are given, we select a small subset of elements from \cref{conj a not 0} and show that they generate $\gr I_x$. In the following part of this section, we present supporting evidence for this conjecture. 

From \cref{power to eh}, we have $y_i^{2k}\in T_{\lambda(b)}$, so $\sco(y_i^{2k})= y_i^{2k+1}\in \gr I_x$. 
\begin{pro}\label{bi=1}
    \cref{conj a not 0} holds when $b_1=...=b_k= 1$. A set of generators of $\gr I_x$ consists of the following elements.
    \begin{itemize}
        \item $y_i^{2k+1}$, $\sum_{i=1}^{n} y_i^{2j}$, $1\leqslant j \leqslant n$.
        \item $y_K= \prod_{l\in K} y_l$ for $K\subset \{1,2,...,n\}$, $ |K|= k+ 1$.
    \end{itemize}
\end{pro}
\begin{proof}
    Write $I^{n}_k$ for the ideals of $\bC[\fh]= \bC[y_1,y_2,...,y_n]$ generated by the elements $y_K$, $y_i^{2k+1}$, and $\sum_{i=1}^{n} y_i^{2j}$ in the lemma. Let $R_k^{n}$ be the quotient ring $\bC[y_1,y_2,...,y_n]/ I^{n}_k$. For a generic $x\in Z_\fg(\fl)$, the set $W.x$ has $|W|/|W_L|= \frac{n!\times 2^{k}}{(n-k)!}$ points. And $I^{n}_k\subset \gr I_\lambda'$, so $\dim R^{n}_k\geqslant |W|/|W_L|$. Our aim is to show that $$\dim R^{n}_k \leqslant \frac{n!\times 2^{k}}{(n-k)!}.$$ 

    We prove this inequality by induction. When $k= n$, $R^{n}_k= \bC[\fh]/ \bC[\fh]^W_+$, which has dimension $2^{n}\times n!$. When $k= 0$, it is clear that $R^{n}_k= \{1\}$. For $n> k> 0$, the ring $R^{n}_k$ has a filtration by the submodules $F_r= y_n^{r}R^{n}_k$, $0\leqslant i \leqslant 2k$. First, $F_0/F_1$ is a free $R_{k}^{n-1}$-module generated by $1$. Next, we show that the quotient $F_{1}/F_{2}$ has a cyclic $R_{k-1}^{n-1}$-module structure. Consider the surjection $$s: R^{n}_{k}/(y_n)\rightarrow F_{1}/F_{2}, \quad y+(y_n)\mapsto  yy_n+ (y_n^{2}).$$
    
    First, $F_1\subset R^{n}_k$ is annihilated by $y_K$, the kernel of $s$ contains the elements $y_{K-1}$ where $K-1\subset \{1,2,...,n-1\}$ and $|K-1|= k$. Next, we claim that $s(y_1^{2k-1})= 0$, and similarly $s(y_i^{2k-1})= 0$, $1\leqslant i\leqslant n-1$. From \cref{power to eh}, we can rewrite $y_1^{2k-2}$ in $R^{n-1}_{k-1}$ as $\sum^{|K'|= k-1}_{K'\subset\{2,...,n-1\}} y_{K'}^{2}$. Now in $F_1/F_2$, we have $y_1y_ny_{K'}^{2}= 0$ for $|K'|= k-1, L\subset \{2,...,n-1\}$, so $y_1^{2k-1}y_n= 0$. 
    
    The analysis in the previous paragraph shows that the kernel of $s$ contains the elements $y_{K-1}$ and $y_i^{2k-1}$. As a consequence, the map $s$ factors through $R^n_{k}/(y_{K-1}, y_i^{2k-1})\cong R^n_{k}$. Hence, $F_1/F_2$ is a cyclic module generated by $y_n$. Lastly, we have surjections $F_1/F_2\rightarrow F_r/F_{r+1}$ by multiplications by $y_n^{i-1}$. In conclusion, we have
    $$\dim R^{n}_k= \dim F_0/F_1+ \sum_{i=1}^{2k} \dim F_r/F_{r+1}\leqslant \dim  R^{n-1}_k+ 2k\dim R^{n-1}_{k-1}$$
    $$= \frac{(n-1)!\times 2^k}{(n-1-k)!}+ 2k\times \frac{(n-1)!\times 2^{k-1}}{(n-k)!}= \frac{n!\times 2^{k}}{(n-k)!}.$$
\end{proof}
Next, we consider the case where $a\geqslant 1$ and $b_1=2\geqslant b_2\geqslant... \geqslant b_k\geqslant 1$. Write $\lambda= (\lambda_0= 2k+1, \lambda_1, 1,...,1)$ for the dual partition of $(2a, b_1,b_1,...,b_k,b_k)$. Here $\lambda_1= 1+ 2(n-a-k)$ because $n-a-k$ is the multiplicity of $2$ in $(b_1,...,b_k)$. Consider the ideal $I_{n,a,k}$ of $\bC[y_1,...,y_n]$ that is generated by the following elements.
\begin{itemize}
    \item $y_i^{2k+1}$, $\sum_{j=1}^{n} y_i^{2j}$, $1\leqslant i,j \leqslant n$.
    \item $y_K= \prod_{l\in K} y_l$ for $K\subset \{1,2,...,n\}$, $ |K|= n-a+1$.
\end{itemize}
This ideal $I_{n,a,k}$ is the ideal generated by the elements in \cref{conj a not 0}.

\begin{pro}\label{b1=2}
    If $\gr I_x= I_{n,a,k}$ for $a=1$ and $n\geqslant 1$, then $\gr I_x= I_{n,a,k}$ for $a>1$. In particular, if \cref{conj a not 0} is true for $a=1$ and $b_1=2$, it is true for $a>1$ and $b_1=2$. 
\end{pro}
\begin{proof}
    Write $R_{n,a,k}$ for the quotient ring $\bC[y_1,...,y_n]/ I_{n,a}$. We prove by induction that $$\dim R_{n,a,k}\leqslant \frac{n!2^n}{a!2^a(2!)^{n-a-k}}= \frac{n!\times 2^k}{a!}.$$ 
    The latter number can be rewritten as   
    $$\frac{n!\times 2^k}{a!}= \frac{(n-1)! 2^k}{(a-1)!}+ 2(n-a-k)\times \frac{(n-1)! 2^k}{a!}+ 2(2k-n+a)\times\frac{(n-1)! 2^{k-1}}{a!}$$
    $$= \dim R_{n-1,a-1,k}+ (\lambda_0-\lambda_1)\dim R_{n-1,a,k}+ (\lambda_1-1)\dim R_{n-1,a,k-1}.$$
    The ring $R_{n,a,k}$ has a decreasing filtration $F_i= (y_n)^i$ for $0\leqslant i\leqslant 2k+1=\lambda_0$. First, $F_0/F_1$ is a cyclic $R_{n-1,a-1,k}$-module. Next, for $1\leqslant r\leqslant 2\lambda_1-1$, the quotients $F_r/F_{r+1}$ are cyclic $R_{n-1, a,k}$-modules. If $\lambda_0= \lambda_1$, or equivalently $b_k=2$, then we are done. 
    
    If $\lambda_0> \lambda_1$, then $b_k=1$. To conclude the proof, we show that $F_r/F_{r+1}$ are cyclic $R_{n-1,a,k-1}$-modules when $r\geqslant \lambda_1$. In particular, we need to prove that $y_i^{2k-1}y_n^{r}$ belongs to $I_{n,a,k}+ (y_1^{r+1})$. Similarly to the proof of \cref{square Tanisaki}, we prove the containment by introducing certain polynomials and work with their coefficients.

    By \cref{power to eh}, we can replace $y_1^{2k-1}$ in $R_{n-1,a,k-1}$ by
    \begin{equation}\label{magic substitution}
        y_1.y_1^{2k-2}= y_1(\sum_{H\subset \{2,...,n-1\}, |H|= k-1} y_H^{2}).
    \end{equation}
    Consider the polynomial $P_{n-1}= (t-y_1)\prod_{i=2}^{n-1} (t^2- y_i^{2})$, write it as $\sum_{d=0}^{2n-1} c_dt^d$. Then the right side of (\ref{magic substitution}) is $c_{2n-2k-2}$. We claim that $y_n^r c_{2n-2k-2}$ is $0$ in $R_{n,a,k}$. Consider the polynomial $P_{n}= (t-y_1)\prod_{i=2}^{n} (t^2- y_i^{2})$ in $R_{n,a,k}$. By \cref{power to eh}, this polynomial can be rewritten in $R_{n,a,k}$ as
    $$(t-y_1)(\sum_{d=0}^{n-1}(-1)^{n-1-d} t^{2d}y_1^{2n-2-2d})= \sum_{d=0}^{2n-1} (-1)^{n-d}t^{d}y_1^{2n-1-d}.$$
    The coefficient of $t^d$ in $P_n$ are $0$ in $R_{n,a,k}$ for $d\leqslant 2n- 2k-2$ because $y_1^{2k+1}$ is $0$ in $R_{n,a,k}$. Using $P_n= (t^{2}-y_n^{2})P_{n-1}$, we obtain $c_d= y_n^{2}c_{d-2}$ for $d\leqslant 2n- 2k-2$. Recall that $r\geqslant \lambda_1= 1+ 2n-2a-2k$, so $y_n^r c_{2n-2k-2}= y_n^{r-(2n-2a-2k)}c_{2a-2}$. The coefficient $c_{2a-2}$ is $\sum_{|H|= n-a} y_H^{2}$, $H\subset \{1,...,n\}$. Since $r-(2n-2a-2k)\geqslant 1$, each monomial in $y_n^{r-(2n-2a-2k)}c_{2a-2}$ consists of precisely $n-a+1$ distinct $y_i$. Therefore, $y_n^{r-(2n-2a-2k)}c_{2a-2}\in I_{n,a,k}$ by the definition of $I_{n,a,k}$.
    
\end{proof}
We do not have proof of \cref{b1=2} for $a=1$, but the equality of dimension can be verified by using computers. A crucial obstacle to proving results like \cref{conj a not 0} is that the technique of using induction and filtration does not work when $a<b_1$. A particular reason is that the number $d_l'$ in \cref{conj a not 0} depends on the index $i$ such that $b_i>a>b_{i+1}$. We give a simple example that explains why we have to assume that $a>1$ in the proof of \cref{b1=2}. 
\begin{rem}\label{why induction fails}
    Consider $a=1, b_1=2$ and $n=3$. The ideal $\gr I_x\subset \bC[y_1,y_2,y_3]$ is generated by $y_1^{2}+y_2^{2}+ y_3^{2}$, $y_1y_2y_3$ and $y_i^{3}$. The quotient ring $\bC[y_1,y_2,y_3]/ \gr I_x$ has dimension $12$, the dimensions of the associated pieces $F_r/F_{r+1}$ are $5,4$ and $3$ for $r=0,1$ and 2, respectively. For an inductive proof to work, the expectation is that each graded piece has dimension $4$.
    
    In particular, $F_0/F_1= \bC[y_1,y_2]/(y_1^{2}+y_2^{2}, y_1^{3}, y_2^{3})$ is not a module over $\bC[y_1,y_2]/(y_1^{2}, y_2^{2})$, the quotient ring for $a=0, b_1=2$. In summary, $I_{n,0,k}$ contains elements $y_i^{2k}$, but $I_{n,1,k}$ contains only $y_i^{2k+1}$. As a consequence, the first graded piece $F_0/F_1$ of $R_{n,1,k}$ is not a module over $R_{n,0,k}$.

\end{rem}

\section{Functions on scheme theoretic intersections}
In this section, we write $\lambda$ for the partition of $e$. We write $I_e$ or $I_\lambda$ for the defining ideal of the scheme-theoretic intersection $\overline{G.e}\cap \fh$ in $\bC[\fh]$. This section studies the weak flatness condition for the orbit $\bO_e= G.e$ of the form $\Ind_{\fl}^{\fg}\{0\}$ for $L$ is a Levi subgroup of $G$. In particular, we discuss when the containment $I_\lambda\subset  \gr I_x$ is an equality where $x$ is a generic element of $\fZ(\fl)$. When $\fg$ and $\lambda$ are specified, we often write $\bO_\lambda$ for the orbit $\bO_e$.

\subsection{Some linear algebra}
We first recall the Pfaffian and the 'symplectic Pfaffian'. They will play a crucial role in determining the ideals $I_\lambda$.
\begin{defin}\label{Pfaffian}
    Consider an element $Y= (y_{ij})$ of $\fso_{2n}= \fso(V)$. The \textit{Pfaffian} $\pf(Y)$ (defined up to a sign) is a polynomial in terms of $y_{ij}$ so that $\pf(Y)^{2}= \det(Y)$. 
\end{defin}

    Choose a basis of $V$ so that the symmetric bilinear form is given by $\Id_{2n}$. An explicit formula is given by
    $$\pf(Y)= \frac{1}{2^{n}n!}\sum_{\sigma\in S_{2n}} sgn(\sigma) \prod_{i=1}^{n} y_{\sigma(2i-1), \sigma(2i)}$$
    where $S_{2n}$ is the symmetric group of size $2n$ and $\sgn(\sigma)$ is the signature of the permutation.

    Next, let $V$ be a symplectic vector space of dimension $2n$. Choose a basis $v_1,...,v_{2n}$ of $V$. Assume that the symplectic form of $V$ is given by a non-singular skew-symmetric matrix $J$, i.e., $\langle u, v \rangle= u^\intercal Jv$ for $u,v\in V$. Consider a matrix $X\in \End(V)$ that satisfies $JXJ^{-1}= X^\intercal$. Then $X$ satisfies the following property. 
    $$\langle u, Xv\rangle= u^\intercal JXv= u^\intercal X^\intercal Jv= (Xu)^\intercal Jv= \langle Xu, v\rangle= - \langle v, Xu\rangle$$ 
    As a consequence, the matrix $A_X=(a_{ij})$ for $a_{ij}= \langle v_i, Xv_j\rangle$ is skew-symmetric.
\begin{defin}\label{symplectic Pfaffian}
    The \textit{symplectic Pfaffian} $\sPf(X)$ is defined as $\pf(A_X)$. It is a polynomial in terms of the entries of $X$ and satisfies $\sPf(X)^2= \det(X)$. (see, e.g., \cite[Proposition 2.9]{ichizuka2024})
\end{defin}
\begin{exa}
    Consider $V= \bC^4$, $J$ is given by the block matrix $\begin{pmatrix}
0 & \Id_2 \\
-\Id_2 & 0 
\end{pmatrix}$. Then $X$ and $A$ takes the forms
$$
X= \begin{pmatrix}
x_{11} & x_{12} & 0 & x_{14}  \\
x_{21} & x_{22} & -x_{14} & 0 \\
0 & x_{23} & x_{11} & x_{21} \\
-x_{23} & 0& x_{12} & x_{22}\\
\end{pmatrix},\quad \quad  A= \begin{pmatrix}
0 & -x_{23} & -x_{11} & -x_{21}  \\
x_{23} & 0 & -x_{12} & -x_{22} \\
x_{11} & x_{12} & 0 & x_{14} \\
x_{21} & x_{22}& -x_{14} & 0\\
\end{pmatrix}.
$$
We have $\sPf(X)= \pf(A)= -x_{23}x_{14}-x_{11}x_{22}+ x_{12}x_{21}$, and $\det(X)= (x_{23}x_{14}+x_{11}x_{22}- x_{12}x_{21})^2$.
\end{exa}    
\begin{rem}\label{spf is better}
    The symplectic Pfaffian possesses a more favorable property compared to the Pfaffian. Consider $X$ so that $JXJ^{-1}= X^\intercal$. Then the matrix $(t\Id_{2n}-X)$ satisfies the same condition and hence has a symplectic Pfaffian. Consequently, the characteristic polynomial of $X$ is a perfect square. Similarly, suitable principal minors of $(t\Id_{2n}-X)$ are perfect squares. We do not have the same result for the Pfaffian and orthogonal matrices.
\end{rem}

Next, we will choose specific bilinear forms to work with elements of $\fg$ in matrix form. Consider the two matrices
$$J_o= \begin{pmatrix}
0 & 1 \\
1 & 0 
\end{pmatrix}, \quad \quad J_s= \begin{pmatrix}
0 & 1 \\
-1 & 0 
\end{pmatrix}.$$
From this point, for an orthogonal (resp. symplectic) vector space $V$ of dimension $2n$, we choose a basis such that the bilinear form is given by the block diagonal matrix $J_o^{n}= \diag(J_o,..., J_o)$ (resp. $J_s^{n}= \diag(J_s,..., J_s)$). Similarly, for an orthogonal vector space $V$ of dimension $2n+1$, the form is given by the matrix $\diag(J_o,...,J_o, 1)$. We write $\fg(V)$ for the corresponding classical Lie algebra.

We use $y_{ij}$ for the coordinate function of $\fg$. Choose the Cartan subalgebra $\fh$ as the diagonal matrices. For notational convenience, we write $y_i$ for $y_{2i-1, 2i-1}$. Thus, we have $\bC[\fh]= \bC[y_1,...,y_n]$. Write $q$ for the quotient map $q: \bC[\fg]\xrightarrow{/(y_{ij})_{i\neq j}} \bC[\fh]$.

\subsection{Cases $\fg$ of type C}
Our main result for $\fg= \fsp_{2n}$ is the following theorem.
\begin{thm}\label{main sp}
    Write $\fl$ as $\fsp_{2a}\times \prod_{i=1}^{k} \fgl_{b_i}$. The map in (\ref{flat}) is an isomorphism if and only if $a=0$. When $a=0$, an explicit description of $\bC[\overline{\bO}_e\cap \fh]$ is given in \cref{square Tanisaki}.
    
    In other words, the weak flatness condition holds for $\bO_e= \Ind_{\fl}^{\fg}(\{0\})$ if and only if $\fl$ does not have a nontrivial $\fsp$ factor. 
\end{thm}
The proof of this theorem is given in \cref{flat sp} and \cref{failedflat sp}. We will need the following technical lemma. Consider $Y\in \fsp_{2n}$. 

Consider $X=Y^2$ and write $M$ for the matrix $(t\Id_{2n}-X)$. Let $L$ be a subset of $\{1,...,n\}$, and let $L'\subset \{1,...,2n\}$ be the set $\{2i, 2i-1, i\in L\}$.
\begin{cor}\label{minor are square}
   The principal minor of $M$ with respect to $L'$ takes the form $\sPf^{2}_L(t,y_{ij})$ for some polynomial $\sPf_L$. Moreover, the quotient map $q_t: \bC[\fg][t]\xrightarrow{/(y_{ij})_{i\neq j}} \bC[\fh][t]$ sends $\sPf_L$ to $\prod_{i\in L} (t- y^{2}_{ii})$.
\end{cor}
\begin{proof}
    Because $Y\in \fsp_{2n}$, we have $J_s^{n}YJ_s^{n}= -Y^\intercal$. Therefore, $(J_s^{n}YJ_s^{n})^2= (Y^\intercal)^2$, and equivalently, $J^{n}_sX(J^{n}_s)^{-1}= X^\intercal$ because $(J_s^{n})^2= -\Id_{2n}$. Let $M_L$ denote the submatrix of $(t\Id_{2n}-X)$ formed by selecting rows and columns with indices in $L'$. Then $M_L$ satisfies $J^{l}_sM_L(J^{l}_s)^{-1}= M_L^\intercal$ where $l= |L|$. Hence, $\det(M_L)$, the principal minor of $M$ with respect to $L'$ can be written as $\sPf(M_L)^{2}$ (see \cref{symplectic Pfaffian}). 
    
    The symplectic Pfaffian $\sPf(M_L)$ is a polynomial in $t$ and the entries $x_{i,j}$ of $X$. Since the entries of $X$ are quadratic polynomials in the entries $y_{i,j}$ of $Y$, $\sPf(M_L)$ can be considered as a polynomial $\sPf_L$ in $t$ and $y_{ij}$. To determine $q_t(\sPf_L)$, we can assume that $Y$ is diagonal and get $q_t(\sPf_L)= $ $\prod_{i\in L} (t- y^{2}_{ii})$.
\end{proof}
\begin{pro}\label{flat sp}
    If $\fl$ does not have a nontrivial $\fsp$ factor then $I_\lambda= \gr I_x$. In other words, $(\ref{flat})$ is a bijection.
\end{pro}
\begin{proof}

It is enough to show that $\gr I_x\subset I_\lambda$. Assume that $\fl= \prod_{i=1}^{k} \fgl_{b_i}$, $b_1\geqslant...\geqslant b_k$. Write $\lambda= (\lambda_0\geqslant ...\geqslant \lambda_m)$ for the dual partition of $(b_1,b_1,...,b_k,b_k)$. Note that $\lambda$ is the patition of $e$, and all members of $\lambda$ are even. Write $\bO_\lambda$ for the orbit $G.e$. To prove \cref{main sp} in this case, we show that the ideal $I_\lambda$ contains the generators of $\gr I_x$ from \cref{square Tanisaki}. In what follows, we will identify the coordinate functions $y_{ii}$ with the coordinate functions $y_i$ of $\fh$ for $1\leqslant i\leqslant n$.

    Consider $Y\in \overline{\bO}_\lambda$, then $X= Y^2$ is a nilpotent matrix in $\fgl_{2n}$. The partition of $X$ is $\lambda'= (\lambda'_0,...,\lambda'_{2m+1})$ where $\lambda'_{2i}= \lambda'_{2i+1}= \lambda_i/2$. For each $l\leqslant n$, let $p_{l}(\lambda')= \lambda'_{2n-2l}+ \lambda'_{2n-2l+1}+...+\lambda'_{2m+1}$. Note that $p_{l}(\lambda')= d_{l}(\lambda)$ in the notation of Section 3.1.
    
    By \cite[Lemma 1]{Tanisaki}, any $2l$-minor of the matrix $(t\Id_{2n}- X)$ is divisible by $t^{p_l(\lambda')}$. Let $L$ be a subset of $\{1, \ldots, n\}$ with cardinality $l$. Let $L'\subset \{1,...,2n\}$ be $\{2i, 2i-1, i\in L\}$. By \cref{minor are square}, the principal minor of $(t\Id_{2n}- X)$ corresponding to the set $L'$ is the square of a polynomial $\sPf_L(t, y_{ij})$. Then $\sPf_L$ is divisible by $t^{p_l(\lambda')/2}$. In other words, the coefficients of $t^d$ of $\sPf_L$ for $d< t^{p_l(\lambda')/2}= d_l(\lambda)/2$ belong to the defining ideal of $\overline{G.e}$ in $\fg$. Using the description of $\sPf_L$ in \cref{minor are square}, we obtain that the coefficients of $t^d$ in $\prod_{i\in L} (t-y_i^{2})$ for $d< d_l(\lambda)/2$ belong to $I_\lambda$, the defining ideal of $\overline{\bO}_\lambda\cap \fh$ in $\fh$. These coefficients are precisely the generators of $\gr I_x$ in \cref{square Tanisaki}. 

\end{proof}

A consequence of \cref{flat sp} and \cref{square Tanisaki} is that when $\lambda= (2^n)$, the ideal $I_\lambda$ is generated by $(y_i^{2})$. This description will be utilized in proving the next proposition.

\begin{pro} \label{failedflat sp}
    Consider $\fl= \fsp_{2a}\times \prod_{i=1}^{k} \fgl_{b_i}$, $a>0$. The map in (\ref{flat}) is not a bijection. 
\end{pro}

\begin{proof}
    From the conditions in Section 3.2, we have $y_1y_2...y_{b_1+...+b_k+1}\in \gr I_x$. To prove the proposition, we show that $y_1y_2...y_{b_1+...+b_k+1}$ is not an element of $I_\lambda$. Recall that we write $c_1$ for $b_1+...+b_k$. As $a>0$, we have $c_1+1\leqslant n$. Consider a matrix realization $\fsp_{2n}= \fsp(V)$ where $V=\bC^{2n}$. We choose a basis $v_1,...,v_{2n}$ of $V$ such that the symplectic form is given by $J_s^{n}$. Write $V'$ for the subspace Span$(v_1,...,v_{2c_1+2})\subset V$. We then have an embedding 
$$i_{c_1}: \fsp_{2c_1+2}= \fsp(V')\hookrightarrow \fsp(V)= \fsp_{2n}.$$
Consider the nilpotent orbit $\bO_{(2^{c_1+1})}$ in $\fsp_{2c_1+2}$ and its image under $i_{c_1}$. We see that 
$$i_{c_1}(\overline{\bO}_{(2^{c_1+1})})\subset \overline{\bO}_{(2^{c_1+1}, 1^{2n-2c_1-2})}\subset \overline{\bO}_\lambda.$$ 
This composition of embedding induces a surjection between functions on scheme theoretic intersections $$\bC[\fh\cap \overline{\bO}_{\lambda}]\twoheadrightarrow \bC[\fh'\cap \overline{\bO}_{(2^{c_1+1})}]= \bC[y_1,...,y_{2c_1+2}]/(y_i^{2})_{i\leqslant 2c_1+2}.$$
The element $y_1...y_{2c_1+2}$ has a nonzero image in $\bC[y_1,...,y_{2c_1+2}]/(y_i^{2})_{i\leqslant 2c_1+2}$, so it is nonzero in $\bC[\fh\cap \overline{\bO}_{\lambda}]$. 
\end{proof}

\subsection{Cases $\fg$ of type B, D}


Consider an orthogonal vector space $V$ of dimension $2n$ or $2n+1$, with its bilinear form represented by $\diag(J_o,...,J_o)$ or $\diag(J_o,...,J_o,1)$. For $l\leqslant n-1$, let $X_l$ be the closed subvariety of $\fso(V)$ consisting of elements with ranks at most $2l+1$. Write $I_l$ for the defining ideal of $X_l$. Recall that we write $y_i$ for $y_{2i-1,2i-1}$.
\begin{lem} \label{Pfaffian lemma}
    The ideal $I_l+ (y_{ij})_{i\neq j}$ contains the element $y_1y_2...y_{l+1}$. As a consequence, it contains all the elements $y_L$ with $L\subset \{1,2,...,n\}$, $|L|= l+1$.
\end{lem}
\begin{proof}   
     We first prove the lemma when $\dim V= 2n$ and $l=n-1$. Write $G$ for $\SO(V)= \SO_{2n}$. For $A\in X_{n-1}$, we have $\rk A\leqslant 2n-1$, so $\det(A)= 0$. Consequently, $\pf(A)= 0$. By the Chevalley restriction theorem (see, e.g., \cite[Section 3.1]{CG}), we have an isomorphism $r: \bC[\fg]^G\cong \bC[\fh]^W$. The Pfaffian (up to a sign) is sent to the product $y_1...y_n$ under this isomorphism. Therefore, the ideal $I_{2n-1}+ (y_{ij})_{i\neq j}$ contains $y_1...y_n$.

      Consider $A\in \fso(V)$. With respect to our bilinear form, the submatrix $A_{\leqslant 2l+2}= (y_{ij})_{i,j\leqslant 2l+2}$ can be viewed as an element of $\fso_{2l+2}$. And for $A\in X_l$, $\rk A_{\leqslant 2l+2}\leqslant 2l+1$. From the result of the previous paragraph for $\fso_{2l+2}$, the Pfaffian $\pf_{\leqslant 2l+2}$ of $A_{\leqslant 2l+2}$ vanishes on $X_l$. Now $\pf_{\leqslant 2l+2}$ takes the form $y_1y_2...y_{l+1}+ f'$ with $f'\in (y_{ij})_{1\leqslant i\neq j\leqslant 2l+2}$, so $y_{1}...y_{l+1}\in I_{l}+ (y_{ij})_{i\neq j}$. Because the ideal $I_{l}+ (y_{ij})_{i\neq j}$ is stable under the action of $G$, we get $y_L\in I_{l}+ (y_{ij})_{i\neq j}$ for $L\subset \{1,2,...,n\}$, $|L|= l+1$.
\end{proof}
 
As mentioned in the Introduction, we have a conjecture about the weak flatness condition for orthogonal Lie algebras $\fg$, see \cref{weakflatso}. In the following, we verify this conjecture in certain cases.
\begin{pro}\label{good case so} Write $\fl\subset \fso_{2n}$ (resp. $\fso_{2n+1})$ as $\prod \fso_{2a}\times \prod_{i=1}^{k} \fgl_{b_i}$ (resp. $\prod \fso_{2a+1}\times \prod_{i=1}^{k} \fgl_{b_i}$).
\begin{enumerate}
    \item The map in (\ref{flat}) is an isomorphism when $b_1=...=b_k= 1$.
    \item If the map in (\ref{flat}) is an isomorphism for $a=1, b_1=2$, then it is an isomorphism for $a>1, b_1=2$
\end{enumerate}
\end{pro}
\begin{proof}
    Consider a generic point $x\in \ft_e$. In Section 2, we discussed the defining ideal $I_x$ of $W.x \subset \fh$, where $W$ is a Weyl group of type B. Let $W_D\subset W$ denote the Weyl group of type D. If any coordinate of $x$ is equal to $0$, the sets $W.x$ and $W_D.x$ coincide. Consequently, for $a\geqslant 1$, we can assume $\fg= \fso_{2n+1}$ for the proof. Thanks to \cref{bi=1} and \cref{b1=2}, it suffices to demonstrate that the ideals $I_\lambda$ contain the generators listed there. Since $I_\lambda$ evidently includes the symmetric polynomials of $y_1^{2},...,y_n^{2}$, we verify the inclusion of the additional generators.
    \begin{enumerate}
        \item When $b_1=...=b_k=1$, we see that $\lambda= (2k+1,1^{2a})$. Consider $A\in \overline{G.e}$. Then $A^{2k+1}=0$, we get $y_i^{2k+1}\in I_\lambda$. Additionally, since the rank of $A$ does not exceed $2k$, it follows from \cref{Pfaffian lemma} that $y_{K}\in I_\lambda$ for $|K|= k+1$.
        \item Consider $a\geqslant 1$ and $b_1= 2$. Then $\lambda= (2k+1,2(n-a-k)+1,1^{2a-1})$ and $2k+1\geqslant 2(n-a-k)\geqslant 1$. Consider $A\in \overline{G.e}$. Then $A^{2k+1}=0$, implying $y_i^{2k+1}\in I_\lambda$. The rank of $A$ is less than or equal to $2k+2(n-a-k)= 2(n-a)$. Hence, we get $y_{K}\in I_\lambda$ for $|K|= n-a+1$ by \cref{Pfaffian lemma}.
    \end{enumerate}
\end{proof}
We have seen the failure of weak flatness in \cref{failedflat sp} for $\fg= \fsp_{2n}$. The argument is that the ideal $I_\lambda$ for $\fg= \fsp_{2n}$ does not contain the element of types $y_1...y_i$. Heuristically, the explanation is that we do not have a version of the Pfaffian for $Y\in \fsp_{2n}$. Similarly, the failure of the weak flatness condition for $\fso$ often comes from the absence of a "symplectic Pfaffian" version for $Y^2$, with $Y \in \fso_{2n}$ (resp. $\fso_{2n+1}$). This phenomenon is illustrated in the following example.
\label{should I actually write something more general. Like I can kill any 3 columns. No, should still give example.}
\begin{exa}\label{example 333}
    Consider $\fl= \fso_3\times \fgl_3\subset \fso_9$, then $\lambda= (3,3,3)$. A generic point $x\in Z_\fg(\fl)$ takes the form $(t_1,t_1,t_1,0)\in \bC^4$. For $y\in W.x$, the polynomial $(z^2-y_1^{2})(z^2-y_2^{2})$ is divisible by $z^2- t_1^{2}$, so $y_1^{2}y_2^{2}\in \gr I_x$. Consider $Y\in \overline{\bO}_\lambda$, then $\rk Y^{2}\leqslant 3$. If $Y^2$ had a symplectic Pfaffian, it would imply that $y_1^{2}y_2^{2}\in \gr I_x$. However, this is not the case. In fact, we can show $y_1^{2}y_2^{2}\notin I_\lambda$ as follows.

    We have a surjection $\bC[y_1,...,y_4]/I_{(3,3,3)}\twoheadrightarrow \bC[y_1,...,y_4]/I_{(3,3,1,1,1)}$. From \cref{good case so} and \cref{b1=2}, we have a complete description of $I_{(3,3,1,1,1)}$. It is generated by $y_i^{5}$, $\sum_{i=1}^{5} y_i^{2l}$, and $y_iy_jy_k$. It can be checked that $y_1^{2}y_2^{2}$ is not $0$ in the quotient $\bC[y_1,...,y_4]/I_{(3,3,1,1,1)}$. Therefore, $y_1^{2}y_2^{2}\notin I_\lambda$.

\end{exa}

\section{Surjectivity of the pullback map}
In this section, the notation for objects on the side of $\fg^\vee$ will always carry a superscript $^\vee$ except the common Weyl group $W$ of $\fg$ and $\fg^\vee$. Consider $\fg^\vee$ of type $B,C$, or $D$. Consider $e^\vee$ nilpotent so that $e^\vee$ is regular in a Levi subalgebra $\fl^\vee \subset \fg^\vee$. Write $\cB^\vee$ and $\spr$ for the flag variety of $\fg^\vee$ and the Springer fiber over $e^\vee$. 

Let $T_\ev\subset G$ be the connected centralizer of $\lv$; it is a torus. By the Jacobson-Morozov theorem, we have an $\fsl_2$-triple $(e^\vee, h^\vee, f^\vee)\subset \fg^\vee$. Let $T_\hv$ be the one-dimensional torus in $G$ such that $Lie(T_\hv)= \bC \hv$. Let $V$ be the standard representation of $\fg^\vee$. Both $T_\ev$ and $T_\hv$ naturally act on $V$, $\bv$ and $\spr$. 

We want to study the surjectivity of the pullback map $i^*: H^*(\bv)\rightarrow H^*(\spr)$ via its equivariant versions $i^{*}_T: H^*_{T}(\bv)\rightarrow H^*_{T}(\spr)$ for various tori $T$ that act. For this purpose, we recall some standard results in equivariant cohomology.
\subsection{Basic results in equvariant cohomology and applications}
Consider a torus $T$ acting on a variety $X$. This section consists of several results on equivariant cohomology $H_{T}^*(X)$ that we often use in the paper. When $X$ is a point, we have $H_{T}^*(\pt)$ is $\bC[\ft]$, the algebra of functions on the Lie algebra of $T$. Equivalently, the ring $H_{T}^*(\pt)$ can be realized as the symmetric algebra of the group of characters of $T$. 

In general, $H_{T}^*(X)$ is defined as the cohomology ring $H^*(\mathbb{E}\times^T X)$ where $\mathbb{E}$ is a certain universal principal $T$-bundle of $X$ (see \cite{Anderson2023EquivariantCI}). With this definition, it makes sense to consider the pullback map in equivariant cohomology. As a consequence, we have a natural map $h_X: H_{T}^*(\pt)$ $\rightarrow H_{T}^*(X)$, and $H_{T}^*(X)$ is a module over $H_{T}^*(\pt)= \bC[\ft]$. A consequence is the following. Consider a $T$-equivariant morphism $f:X\rightarrow Y$. Then $f_{T}^{*}: H^{*}_T(Y) \rightarrow H^{*}_T(X)$ is surjective if and only if $f^{*}: H^{*}(Y) \rightarrow H^{*}(X)$ is surjective by graded Nakayama lemma.

\subsubsection{Localization theorems and specialization of pullback} \label{fixed point pullback}
In this paper, we will only consider cases with the following features.
\begin{enumerate}
    \item $T$ acts on $X$ with finitely many fixed points.
    \item $H^*_{T}(X)$ is a free module over $H^*_{T}(\pt)$.
\end{enumerate}
In particular, the action of $T_\ev\times T_\hv$ on $\spr$ and $\bv$ satisfies these conditions, see, e.g., \cite{DLP}. With these conditions, the localization theorem in equivariant cohomology (see, e.g. \cite[Section 5]{Anderson2023EquivariantCI}) implies that the pullback map $\iota^*: H_{T}^*(X)\rightarrow H_{T}^*(X^T)$ is an injection. Moreover, it becomes an isomorphism after localizing at certain multiplicative subset $S\subset \bC[\ft]$ that is generated by homogeneous polynomials. When $T\cong \bC^\times$, we have $\bC[\ft]= \bC[t]$, the polynomial ring in one variable. In this case, we can always choose $S$ as the set generated by $t$.

Let $\ft$ be the Lie algebra of $T$, and $t\in \ft$. The equivariant cohomology $H^*_{T}(X)$ is an algebra over $\bC[\ft]$. Write $H^{*}_T(X)_t$ for the specialization of $H^*_{T}(X)$ at $t$. Assume that $X$ is smooth. $T$ acts on $X$, so we can regard $t$ as a vector field of $X$. Write $X^t\subset X$ for the fixed point subvariety of the action of this vector field. By the Berline-Vergne localization theorem (\cite[Proposition 2.1]{Berline1985}, \cite[Sections 4,5]{Atiyah1984}), we have the following isomorphism.

$$H^{*}_T(X)_t \cong H^*(X^{t})$$
Note that this isomorphism commutes with the pullback map as follows. Consider a $T$-equivariant morphism $f:X\rightarrow Y$. Then the isomorphism above identifies $f_t^{*}: H^{*}_T(Y)_t \rightarrow H^{*}_T(X)_t$ with the pullback for the fixed-point varieties $H^*(Y^{t}) \rightarrow H^*(X^{t})$.
\begin{rem}
    In what follows, we often apply the Berline-Vergne localization theorem to the flag variety and the Springer fiber $\spr$. Although $\spr$ is not smooth, the isomorphism $H^{*}_T(\spr)_t \cong H^*((\spr)^{t})$ holds for $T= T_\ev$. In particular, we can first apply the theorem to the Slodowy variety $\tilde{S}_\ev$ and then use the fact that $(\tilde{S}_\ev)^t$ contracts to $(\spr)^t$. 
\end{rem}
\subsubsection{Image of the map $i^*_{T_\ev}$}
Consider the pullback map in equivariant cohomology $i^*_{T_\ev}: H^*_{T_\ev}(\bv)\rightarrow H^*_{T_\ev}(\spr)$. Write $K(\ev)$ and $\cA_{T_\ev}$ for the kernel and the image of this map. Thanks to the localization theorem, we can identify $\cA_{T_\ev}$ with the image of the composition $i^*_{T_\ev}: H^*_{T_\ev}(\bv)\rightarrow H^*_{T_\ev}(\spr)\hookrightarrow H^*_{T_\ev}((\spr)^{T_\ev})$. This composition is described in \cite{HMK2024}, we summarize the related results below. 

Consider a symplectic vector space (resp. orthogonal vector space) $V$ of dimensions $2n$ (resp. $2n$ or $2n+1$). Realize $\fg^\vee$ as $\fg^\vee(V)$. Consider a torus $T$ that acts on $V$, preserving the bilinear form. Then $T$ acts on $\bv$, the variety of maximal isotropic flags of $V$. 
A maximal isotropic flag of $V$ is given by $F^{\bullet}= (F^{1}\subset...\subset F^{n})$ where $\dim F^i= i$. On $\bv$, we have the tautological line bundles $\cF^{i}$ that parameterize the pair $(F^\bullet, F^i/F^{i-1})$. Write $y_i$ for $c_1^{T}(\cF^i)$. Write $c_i$ for the Chern class $c_i^{T}(V)\in H^{*}_T(\pt)$. 
\begin{pro}\cite[Proposition 13.3.1]{Anderson2023EquivariantCI}
    Assume that $\fg^\vee$ is of type B or C. We have 
    $$H^{*}_{T}(\bv)= \bC[\ft][y_1,...,y_n]/(e_i(y_{1}^{2},...,y_{n}^{2})- c_{2i}(V))_{i=1,...,n}$$
    in which $e_i(y_{1}^{2},...,y_{n}^{2})$ are the elementary symmetric polynomials of $y_i^{2}$. 
\end{pro}

Let $\fl\subset \fg$ be the Langland duals of $\lv\subset \fg^\vee$. Write $W_L$ for the Weyl group of $\fl$. Consider a generic point $x\in Z_\fg(\fl)$. Recall that we write $x$ as $(t_1,...,t_1,...,t_k...,t_k,0...,0)$ in which $t_i$ appears $b_i$ times and $0$ appears $a$ times. View $x$ as a point of $\ft$. In this section, we will let $x$ vary over the open subset $U= \{(t_1,...,t_k)\subset \ft, t_i\neq t_j\}$. Then it makes sense to identify $t_i$ with the coordinate functions of $\ft_\ev$. View $x= (t_1,...,t_1,...,t_k...,t_k,0...,0)$ as an element of $\prod_{i=1}^{n} \bC[\ft_\ev]$. The orbit $W.x$ that we have considered in Section 2 can be regarded as a subset of $\prod_{i=1}^{n} \bC[\ft_\ev]$. 

From \cite[Section 3]{HMK2024}, the variety $(\spr)^{T_\ev}$ consists of discrete points labeled by $W_L$-cosets in $W$. Hence, we have $H^*_{T_\ev}((\spr)^{T_\ev})= \bigoplus_{|W/W_L|} \bC[\ft_\ev]$. For each $w\in W/W_L$, write $\bC[\ft_\ev]_w$ for the summand of $H^*_{T_\ev}((\spr)^{T_\ev})$ labeled by $w$. We have the following proposition.
\begin{pro}\cite[Section 7]{HMK2024} \label{identify image}
    The composition $$i^*_{T_\ev}(w): H^*_{T_\ev}(\bv)\rightarrow H^*_{T_\ev}((\spr)^{T_\ev})\twoheadrightarrow \bC[\ft_\ev]_w$$ is given by $(y_1,...,y_n)\mapsto w.x= w.(t_1,...,t_1,...,t_k...,t_k,0...,0)$. 
\end{pro}

Consider the composition 
\begin{equation} \label{polyquotient}
    \bC[\ft_\ev][y_1,...,y_n]\twoheadrightarrow H^*_{T_\ev}(\bv)\rightarrow H^*_{T_\ev}((\spr)^{T_\ev}).
\end{equation}
For a point $x\in \ft_\ev$, the specialization of (\ref{polyquotient}) takes the form $i^{*}_x: \bC[y_1,...,y_n]\rightarrow \bigoplus_{|W/W_L|} \bC$. Thanks to \cref{identify image}, for $x$ generic, $i^{*}_x$ is precisely the map induced from the inclusion $W.x\hookrightarrow \bC^{n}$. Therefore, we have the following corollary.
\begin{cor} \label{free}
    Assume that we have a uniform set of generators $P_1,...,P_M\subset \bC[y_1,...,y_n,t_1,...,t_k]$ of $I_x$. The algebra $\cA_{T_\ev}$ is given by $\bC[\fh,\ft]/(P_1,...,P_M)$. The latter is a flat (and hence free) module of rank $|W/W_L|$ over $\bC[\ft_\ev]$. In particular, if \cref{conj a not 0} holds, the image $\cA_{T_\ev}$ is always a free module over $\bC[\ft_\ev]$.
\end{cor}

\subsection{Surjectivity of the pullback map}
Write $\lambda^\vee$ for the partition of $e^\vee$ in $\fg^\vee$. Write $i^*_{\ev}$ for the pullback map $H^*(\bv)\rightarrow H^*(\spr)$. Then $i^*_{\ev}$ is surjective if and only if $i^*_{T_\ev}$ is surjective by the graded Nakayama lemma. We have the following general result. 

Consider $t\in \ft_\ev$. Write $\fl_t^{\vee}$ for the centralizer of $t$ in $\fg^\vee$. We have $\ev\in \fl_t^{\vee}$. Write $\cB(\fl_t^{\vee})$ and $\cB_{\ev}(\fl_t^{\vee})$ for the flag variety of $\fl_t^{\vee}$ and the Springer fiber of $\ev$ with respect to $\fl_t^{\vee}$. Let $W_{L_t^{\vee}}$ be the Weyl group of $\fl_t^{\vee}$. Write $(\bv)^t$ and $(\spr)^t$ for the fixed point varieties of $\bv$ and $\spr$ under the actions of the vector fields induced by $t$. 
\begin{pro}\cite[Proposition 3.12]{HMK2024} \label{fixed point Tt}
    \begin{enumerate}
        \item Both $(\bv)^{t}$ and $(\spr)^{t}$ have $|W/W_{L_t}|$ connected components. They are isomorphic to $\cB(\fl_t^{\vee})$ and $\cB_{\ev}(\fl_t^{\vee})$, respectively.
        \item The closed embedding $i_{T_t^\vee}: (\spr)^{t}\hookrightarrow (\bv)^{t}$ can be realized as $|W/W_{L_t^{\vee}}|$ copies of the natural closed embedding $i_{\ev}(\lv_t): \cB_{\ev}(\fl_t^{\vee})\hookrightarrow \cB(\fl_t^{\vee})$.
    \end{enumerate}    
\end{pro} 
    The Levi subagebra $\fl^{\vee}_t\subset \fg^\vee$ has a decomposition of the form $(\fg^{\vee})'\times \prod_{i=1}^{m} \fgl_{t_i}$ where $(\fg^{\vee})'$ is of the same type as $\fg^\vee$. Similarly, we can decompose $e^\vee$ as $((e^{\vee})',e^{\vee}_1,...,e^{\vee}_m)$ for $(e^{\vee})'\subset (\fg^{\vee})'$ and $e^{\vee}_i\subset \fgl_{\lambda_i}$. Note that $(e^{\vee})'$ is regular in $(\fl^\vee)':= \fl^\vee \cap (\fg^\vee)'$. Let $\cB_i$ be the flag variety for $\fg\fl_i$. Because the maps $H^*(\cB_i)\rightarrow H^*(\cB_{e^{\vee}_i})$ are surjective (see, e.g., \cite{deConcini1981}), we have the following result.
\begin{cor}\label{gl surjective}
    Consider $t\in \ft_\ev$. In the notation of Section 4.1.1, consider the specialized pullback map  
    $i^{*}_{t}: H^*_{T_\ev}(\bv)_t\rightarrow H^*_{T_\ev}(\spr)_t.$
    Then $i^{*}_{t}$ is surjective if and only if $(i')^*: H^*((\cB')^\vee)\rightarrow H^*(\cB_{(e^\vee)'})$ is surjective. 
\end{cor}

Recall that we consider $\ev$ is regular in $l^\vee \subset \fg^\vee$. And we write $l^\vee$ as $\fg^\vee(a)\times \prod_{i=1}^{k} \fgl_{b_j}$ where $\fg^\vee(a)$ is of the same type as $\fg^\vee$ with a root system of size $a$. For each $1\leqslant j\leqslant k$, write $\lv_j$ for the algebra $\fg^\vee(a)\times \fgl_{b_j}$. Consider $\lv_j$ as a Levi subalgebra of $\fg^\vee(a+b_j)$. Let $e_j^\vee$ be a regular nilpotent element in $\lv_j$. Let $\bv_j$ be the flag variety of $\fg^\vee(a+b_j)$. Then we have the pullback maps $i^*_{e_j^\vee}: H^*(\bv_j)\rightarrow H^*(\cB_{e_j^\vee})$.

\begin{pro}\label{reduce to 3}
    Assume that $\cA_{T_\ev}$ is a free module over $\bC[\ft_\ev]$. The map $i^*_{\ev}: H^*(\bv)\rightarrow H^*(\spr)$ is surjective if and only if all $i^*_{e_j^\vee}$ are surjective for $1\leqslant j\leqslant k$. 
\end{pro}
\begin{proof}
    Consider the subvariety $Z_{ij}$ of $\ft_\ev$ given by $\{(t_1,...,t_k)\in \ft_\ev, t_i= t_j=0\}$. Let $U$ denote the complement $\ft_\ev \setminus \bigcup_{1\leqslant i<j \leqslant k} Z_{ij}$. Consider the natural injection $\iota: \cA_{T_\ev} \hookrightarrow H^*(\spr)$ between two free $\bC[\ft_\ev]$ modules. Because $\ft_\ev\setminus U$ has codimension $2$ in $\ft_\ev$, the injection $\iota$ is an isomorphism if and only if it is an isomorphism over $U$ (see, e.g., \cite[Proposition 1.6]{Hartshorne1980}). In other words, $i_\ev^{*}$ (equivalently, $i_{T_\ev}^*$) is surjective if and only if $i_t^{*}$ is surjective for $t\in U$.  

    Let $Z_i$ be $\{(t_1,...,t_k)\in \ft_\ev, t_i=0\}$. Then $Z'_i= Z_i\cap U= \{(t_1,...,t_k)\in \ft_\ev, t_i=0, t_j\neq 0 \text{ for } j\neq i\}$. Write $U$ as a disjoint union $U'\sqcup (\bigsqcup_{i=1}^{k} Z'_i)$. Then $U'= \{(t_1,...,t_k)\in \ft_\ev, t_i\neq 0 \text{ for } 1\leqslant i\leqslant k\}$.

    For $t\in U'$, the corresponding Levi $\lv_t\subset \fg^\vee$ is a product of $\fg^\vee(a)$ and $\fgl$ factors. The restriction of $e^\vee$ to the factor $\fg^\vee(a)$ is a regular nilpotent element of $\fg^\vee(a)$, so the corresponding Springer fiber is a single point. By \cref{gl surjective}, we have $i_t^{*}$ is surjective.

    For $t\in Z_j$, the corresponding Levi is $\lv_t\subset \fg^\vee$ is a product of $\fg^\vee(a+b_j)$ and $\fgl$ factors. The restriction of $e^\vee$ to the factor $\fg^\vee(a+b_j)$ is $e_j^\vee$. Therefore, $i_t^{*}$ is surjective if and only if $i^*_{e_j^\vee}$ is surjective.
\end{proof}

\subsection{Cases of Hikita conjecture}
Consider $\ev$ regular in $\lv\subset \fg^\vee$. Recall that we write $G.e= \bO_e\subset \fg$ for the induced nilpotent orbit $\Ind_\fl^{\fg}\{0\}$.
\begin{thm}\label{surjective case}
    The Hikita conjecture for the pair $(S(e^\vee), \Spec (\bC[\bO_e]))$ holds in the following cases.
    \begin{enumerate}
        \item $\fg= \fsp_{2n}$, $\fl$ takes the form $\prod_{i=1}^{l} \fgl_2\times \prod_{j=1}^{n-l} \fgl_1$.
        \item $\fg= \fso_{2n+1}$, $\fl$ takes the form $\fso_3 \times \prod_{i=1}^{l} \fgl_2\times \prod_{j=1}^{n-l-1} \fgl_1$, or $\fso_{2l+1}\times \prod_{j=1}^{n-l} \fgl_{1}$.
        \item $\fg= \fso_{2n}$, $\fl$ takes the form $\fso_4\times \prod_{i=1}^{l} \fgl_2\times \prod_{j=1}^{n-l-2} \fgl_1$, $\fso_2\times \prod_{i=1}^{l} \fgl_2\times \prod_{j=1}^{n-l-1} \fgl_1$, or $\fso_{2l}\times \prod_{j=1}^{n-l} \fgl_1$.
    \end{enumerate}
\end{thm}
\begin{proof}
    For $\fl$ in the theorem, we have shown that the surjection $\bC[\overline{\bO}_e\cap\fh]\twoheadrightarrow \bC[\fh]/ \gr I_x$ is an isomorphism in \cref{flat sp} and \cref{good case so}. It is left to prove the pullback map $i^*_{\ev}: H^*(\bv)\rightarrow H^*(\spr)$ is surjective in these cases.

    By \cref{reduce to 3}, we only need to prove $i^*_{\ev}$ is surjective for the cases where $\fl$ has precisely one factor $\fgl$. In these cases, $\ft_\ev\cong \bC$. Recall that we have a one dimensional torus $T_\hv$ that comes from an $\fsl_2$-triple of $\ev$. Consider the pullback map $i^*_{T_\ev\times T_\hv}: H^*_{T_\ev\times T_\hv}(\bv)\rightarrow H^*_{T_\ev\times T_\hv}(\spr)$. Similarly to the proof of \cref{reduce to 3}, the surjectivity of $i^*_{T_\ev\times T_\hv}$ follows from the following two results. 
    \begin{itemize}
        \item The image of $i^*_{T_\ev\times T_\hv}$ is a free module over $\bC[\ft_\ev,\ft_\hv]$, see \cref{free small cases}.
        \item For $t\in \ft_\ev\times \ft_\hv= \bC^2$ and $t\neq (0,0)$, the specialized map $i^*_{t}: H^*_{T_\ev\times T_\hv}(\bv)_t\rightarrow H^*_{T_\ev\times T_\hv}(\spr)_t$ is surjective, see \cref{fixed pt}. 
    \end{itemize}
    We prove these results by direct computations; the details are given in \cref{Springerfix}.
\end{proof}

\section{Hikita conjecture for spherical orbits in $\fso_N$}
Let $\fg$ be a classical Lie algebra of type B or D. Let $G$ be the corresponding Lie group and let $B$ be a Borel subgroup of $G$. Consider a nilpotent element $e\in \fg$. The orbit $G.e$ is called \textit{spherical} if $B$ has an open orbit in $G.e$. The partitions of spherical nilpotent orbits have been classified in \cite[Tables 1,2]{Panyushev2003}. In type D and type B, they take the forms $(2^{2k}, 1^{N-4k})$ or $(3, 2^{2k}, 1^{N-4k-3})$ for $N= 2n,2n+1$. From \cite[Proposition 6.3.7]{collingwood1993nilpotent}, the special spherical nilpotent orbits are as follows.
\begin{enumerate}
    \item $\fg$ is of type D, the partition of $e$ is $(2^{2k},1^{2n-2k})$, $0\leqslant k\leqslant n$. On the side of $\fg^\vee$ side, we consider $e^\vee$ with the partition $\lambda^\vee=(2n-2k-1, 2k+1)$ or $(2k,2k)$ when $n=2k$.
    \item $\fg$ is of type B, the partition of $e$ is $(3,2^{2k-2},1^{2n-2k})$, $1\leqslant k\leqslant n$. On the side of $\fg^\vee$ side, we consider $e^\vee$ with the partition $\lambda^\vee=(2n-2k-1, 2k+1)$ or $(2k,2k)$ when $n=2k$. 
\end{enumerate}
In this section, we prove that the Hikita conjecture holds for these special spherical nilpotent orbits.
\subsection{Case $\fg$ of type D}
We first discuss the case where $\lambda^\vee$ has two odd rows. Let the partition $\lambda^\vee$ of $e^\vee$ be $(2n-2k-1\geqslant 2k+1)$, then the partition of $e$ is $(2^{2k}, 1^{2n-4k})$. From \cite[Section 6]{collingwood1993nilpotent}, the fundamental group of $\bO_e= G.e$ is trivial, so $\tilde{D}(\bO_\ev)= \bO_e$. An explicit description of $H^*(\spr)$ for $e^\vee\in \fso_{2n}$ is given by \cite[Theorem D]{Stroppel_2018} and \cite[Theorem 1.2]{Ehrig2016} as follows:
$$H^*(\spr) \cong \bC[y_1, y_2,...,y_n]/(y_i^{2}, y_{K}) \quad 1\leqslant i\leqslant n; K\subset \{1,2,...,n\}, |K|= k+1.$$
Recall that we have a matrix realization of $\fso_{2n}= \fso(V)$ with $V=\bC^{2n}$ and $V$ has a basis $v_1,...,v_{2n}$ such that the symmetric biliniear form of $V$ is given by $(J_o,...,J_o)$. We write $y_{ij}$ for coordinate function of $\fso(V)$ and write $y_{i}$ for $y_{2i-1, 2i-1}$. The first main result of this section is the following proposition.
\begin{pro} \label{main two row}
    We have an isomorphism 
    $$\bC[\fh\cap \overline{\bO}_e]\cong \bC[y_1, y_2,...,y_n]/(y_i^{2}, y_{K})$$
    In other words, the Hikita conjecture holds for $e\in \fso_{2n}$ with partitions $(2^{2k},1^{2n-2k})$, $k>0$.
\end{pro}
First, we need a lemma of the case where $e^\vee$ and $e$ are very even. The proof of the proposition is given after this lemma. Let $V'= \Span(v_1,...,v_{4k})\subset V$ with the bilinear form restricted from $V$. Consider $\fg'= \fsp_{4k}$ and let $\fh'$ be the Cartan subalgebra of $\fg'$. Then $\bC[\fh']= \bC[y_1,y_2,...,y_{2k}]$. Consider the partition $\lambda'= (2^{2k})$, it is a very even partition of $4k$. Let $\bO^I$ and $\bO^{II}$ be the two orbits associated with $\lambda'$. The closures $\overline{\bO^I}$ and $\overline{\bO^{II}}$ are normal in $\fso_{4k}$ (see \cite[Section 17] {Kraft1982OnTG}), and they have the same rings of regular functions. 
\begin{lem} \label{very even}
    Let $I_{(2k,2k)}$ be the ideal generated by $y_i^{2}, y_L- y_{\{1,2,...,2k\}\setminus L}$ for $1\leqslant i\leqslant 2k$ and $L\subset \{1,2,...,2k\}$, $|L|= k$. We have a surjection 
    \begin{equation} \label{surjection very even}
        \bC[\fh'\cap \overline{\bO^I}]\twoheadrightarrow \bC[y_1,...y_{2k}]/ I_{(2k,2k)}.
    \end{equation}
\end{lem}
\begin{proof}
    We claim that this surjection is the surjection introduced in (\ref{flat}). Note that $\bO^I$ is birationally Richardson with respect to the Levi subalgebra $\fl= \fgl_{2k}$. We can assume $\fl= \diag(\fgl_{2k}, -\fgl_{2k})$ in $\fg'$. Then a generic element $x'\in Z_{\fg'}(\fl')$ takes the form $\diag(t\Id, -t\Id)$. Let $W'$ be the Weyl group of $\fg'$, and write $I_{x'}$ for the defining ideal of $W'.x'$ in $\fh'$. The set $W'.x'$ consists of elements $(y_1,y_2,...,y_{2k})$ such that $y_i= \pm t$, and $y_1y_2...y_n= t^{2k}$. Hence, $\gr I_{x'}$ contains $I_{(2k,2k)}$. It is straightforward to check that $$\dim \bC[y_1,...y_{2k}]/ I_{(2k,2k)}= \sum_{l=0}^{2k-1} {4k \choose l}+ \frac{1}{2} {4k \choose  2k}= 2^{2k-1}= |W'.x|$$
    Thus $I_{(2k,2k)}= \gr I_{x'}$. From \cite[Proposition 4]{conjugacyclass}, we have a surjection $\bC[\fh'\cap \overline{\bO^I}] \twoheadrightarrow \bC[\fh]/ \gr I_{x'}=  \bC[\fh]/I_{(2k,2k)}$.

\end{proof}
Later, in \cref{very even}, we show that the surjection in the above proposition is an isomorphism, thus justifying the notation $I_{(2k,2k)}$. We now proceed to the proof of \cref{main two row}
\begin{proof}
    It is clear that for a matrix $A\in \overline{\bO}_e$, we have $A^2=  0$ and $\rk A$ $\leqslant 2k$. Similarly to the previous section, we have $y_i^{2}, y_{K}\in I_\lambda$. Hence, $\bC[\fh\cap \overline{\bO}_e]$ is a quotient of $\bC[y_1, y_2,...,y_n]/(y_i^{2}, y_{K})$. To prove the proposition, we first show that the element $y_1y_2...y_{k}$ is nonzero in $\bC[\fh\cap \overline{\bO}_e]$. 

    Let $V'= \Span(v_1,...,v_{4k})\subset V$ with the bilinear form restricted from $V$. Then we have an inclusion 
    $$i_{4k}: \fso_{4k}= \fso(V')\hookrightarrow \fso(V)= \fso_{2n}$$    
    It is clear that $i_{4k}(\overline{\bO^I}), i_{4k}(\overline{\bO^{II}})\subset \overline{\bO}_e$. And $i_{4k}$ sends the diagonal matrices $\fh' \subset \fso_{4k}$ to $(\fh', 0,...,0)\subset \fh$. These closed embeddings induce a surjection on the level of functions:
    \begin{equation} \label{surject to different type}
        i_{4k}^*: \bC[\fh\cap \overline{\bO}_e]\twoheadrightarrow \bC[\fh'\cap \overline{\bO^I}].
    \end{equation}
    Combing with the surjection from \cref{very even}, we have a surjection $$\bC[\fh]/I_\lambda= \bC[\fh\cap \overline{\bO}_e]\twoheadrightarrow \bC[y_1,...y_{2k}]/ I_{(2k,2k)}.$$ 
    The image of $y_1y_2...y_k$ does not belong to $I_{(2k,2k)}$, so it is nonzero in the target. Hence, $y_1y_2...y_k$ does not belong to $I_\lambda$.

    Since we can view ${\bO_e}$ as an $O_{2n}$-orbit of $\fso_{2n}$, the defining ideal of $\overline{\bO}_e$ is $O_{2n}$-stable. We have a subgroup $\tilde{W}$ of $O_{2n}$ such that $\tilde{W}\cong S_n \rtimes (\cyclic{2})^{\oplus n}$ and $\tilde{W}$ acts on $\fh$ as the Weyl group of type $B$. Note that the vector subspace $V_{K}=\Span \{\tilde{W}. (y_1y_2...y_k)\}$ of $\bC[\fh]$ is an irreducible repersentation of $\tilde{W}$ of dimension $n\choose k$ (see \cref{irred} below). Because $I_\lambda$ is $\tilde{W}$-stable, the intersection $I_\lambda \cap V_K$ must be a $\tilde{W}$ submodule of $V_K$, hence it is $V_K$ or $\{0\}$. Now $y_1y_2...y_k\notin I_\lambda$, so $V_K\cap I_\lambda= \{0\}$. 
    
    Similarly, considering the embeddings $\fso_{4l}\hookrightarrow \fso_{2n}$ for $1\leqslant l\leqslant k$, we have $V_{L}\cap I_\lambda= \{0\}$ for $V_{L}=$ Span$\{\tilde{W}. (y_1y_2...y_l)\}$ $\subset \bC[\fh]$. Therefore, $\dim \bC[\fh\cap \overline{\bO}_e]\geqslant \sum_{l=0}^k {n \choose l}$=  $\dim \bC[y_1, y_2,...,y_n]/(y_i^{2}, y_{K})$. This completes the proof of the proposition.   
\end{proof}
\begin{lem}\label{irred}
    $V_{K}=$ Span$\{\tilde{W}. (y_1y_2...y_k)\}$ is an irreducible repersentation of $\tilde{W}$.
\end{lem}
\begin{proof}
    Consider a nontrivial subrepresentation $U$ of $V_K$. Consider $u\in U$ so that $u$ has the least terms as a linear combination of monomials. Fix $1\leqslant i\leqslant n$, if some of these monomials are divisible by $y_i$, write $u_i$ for the sum of them. Then $u_i\in U$ and $u- u_i\in U$, so we need $u_i=u$ or $u_i=0$. Therefore, all terms of $u$ are not divisible by $y_i$ or $u$ is divisible by $y_i$ for $1\leqslant i\leqslant n$. This implies that $u$ is a multiple of a single monomial and $U= V_K$.
\end{proof}

Next, we proceed to consider the case $\lambda^\vee= (2k, 2k)$ in $\fso_{4k}$. There are two corresponding nilpotent orbits. The two corresponding Springer fibers are isomorphic and have the same cohomology rings. Hence, we can write $\spr$ without any confusion. Note that in this case, the Springer fiber $\spr$ is not connected, it has two isomorphic connected componenets $(\spr)^1$ and $(\spr)^0$. By \cite[Theorem 1.2]{Ehrig2016}, we have an explicit description
$$H^*((\spr)^0)\cong H^*((\spr)^1) \cong \bC[y_1, y_2,...,y_{2k}]/(y_i^{2}, y_{L}- y_{\bar{L}}) \quad 1\leqslant i\leqslant n; L\cup \bar{L}= \{1,2,...,2k\}, |L|= |\bar{L}|= k$$
In the notation of \cref{very even}, this equality reads $H^*((\spr)^0)\cong \bC[y_1,...y_{2k}]/ I_{(2k,2k)}$. It makes sense to replace $H^*(\spr)$ in the cohomological side of Hikita conjecture by $H^*((\spr)^0)$.
\begin{pro}
    The surjection (\ref{surjection very even}):
    $$\bC[\fh'\cap \overline{\bO^I}]\twoheadrightarrow \bC[y_1,...y_{2k}]/ I_{(2k,2k)}$$
    in \cref{very even} is an isomorphism. Combining with \cref{main two row}, we have that the Hikita conjecture holds for special spherical nilpotents $e\in \fso_{2n}$.
\end{pro}
\begin{proof}
    In the proof of \cref{main two row}, let $n= 4k+2$ and consider the partition $(2^{2k},1,1)$ of $n$. Let $\bO_{(2^{2k},1,1)}$ be the corresponding nilpotent orbit in $\fso_{4k+2}$. We have showed that $\bC[\overline{\bO}_{(2^{2k},1,1)}\cap \fh]\cong \bC[y_1,...y_{2k+1}]/(y_i^2, y_{K})$ where $y_{K}$ denotes the elements $\prod_{i\in K} y_i$ for $|K|= k+1$. Hence, the surjection (\ref{surject to different type}) reads:
    $$i_{4k}^*: \bC[y_1,...y_{2k+1}]/(y_i^2, y_{K})\twoheadrightarrow \bC[\fh'\cap \overline{\bO^I}], \quad i_{4k}^*(y_{2k+1})= 0, i_{4k}^*(y_j)= y_j \text{ for } j<2k +1.$$
    It is clear that $i_{4k}^*$ factors through $\bC[y_1,...y_{2k}]/(y_i^2, y_{K})$. Combining with the surjection (\ref{surjection very even}), we have a map
    $$p: \bC[y_1,...y_{2k}]/(y_i^2, y_{K}) \xrightarrow{p_1} \bC[\fh'\cap \overline{\bO^I}]\xrightarrow{p_2} \bC[y_1,...y_{2k}]/I_{(2k,2k)}.$$
    Let $W_D^{2k}$ be the Weyl group of $\fso_{4k}$. Note that $p$ is $W_D^{2k}$-equivariant, and the kernel of $p$ is spanned by the elements $W_{D}^{2k}.(y_1...y_k- y_{k+1}...y_{2k})$, in other words, $y_L- y_{\bar{L}}$. This kernel is an irreducible representation of $W_D^{2k}$, so either $p_1$ or $p_2$ is an isomorphism of $W$-representations.

    We prove the proposition by contradiction. Assume that $p_1$ is an isomorphism. Note that a similar construction works for $\bO^{II}$. And we have an isomorphism $p_1': \bC[y_1,...y_{2k}]/(y_i^2, y_{K}) \xrightarrow{p_1'} \bC[\fh'\cap \overline{\bO^{II}}]$ by composing $p_1$ with the map $y_1\mapsto -y_1, y_j\mapsto y_j$  for $1<j\leqslant 2k$. Now, consider the scheme-theoretic intersection $\overline{\bO^{II}}\cap \overline{\bO^{I}}$. This scheme is reduced by \cite[Section 17.3]{Kraft1982OnTG}, and it is the closure of the nilpotent orbit $\bO_{(2^{2k-2},1^4)}$ associated to the partition $(2^{2k-2},1^4)$. Therefore, the defining ideal of $\bC[\overline{\bO}_{(2^{2k-2},1^4)}\cap \fh']$ in $\bC[\fh']$ is the sum of the defining ideals of $\bC[\fh'\cap \overline{\bO^{I}}]$ and $\bC[\fh'\cap \overline{\bO^{II}}]$ in $\bC[\fh']$, which is $(y_i^{2}, y_{K})$, $|K|= k+1$. Nevertheless, from \cref{main two row}, we have 
    $$\bC[\bO_{(2^{2k-2},1^4)}\cap \fh]\cong \bC[y_1,...,y_{2k}]/(y_i^2, y_{K'})\neq \bC[y_1,...,y_{2k}]/(y_i^2, y_{K})$$
    where $K'$ runs over subsets of $\{1,2,..,2k\}$ with cardinalities $k$ while $K$ runs over subsets of $\{1,2,..,2k\}$ with cardinalities $k+1$. This contradiction means that $p_2$ must be an isomorphism.
\end{proof}
\subsection{Case $\fg$ of type B}
The context of the Hikita conjecture for special spherical nilpotent orbits $G.e\subset \fso_{2n+1}$ in this section have the following new features.
    \begin{enumerate}
    \item For $\ev\in \fsp_{2n}$ with partition $(2n-2k\geqslant 2k)$. The refined BVLS duality sends $G^\vee.e^\vee$ to a double cover of $G.e$ where $e$ has partition $\lambda= (3,2^{2k-2},1^{2n-4k+2})$. Write $\tilde{\bO}_\lambda$ for this double cover.
    \item For $\ev\in \fsp_{2n}$ with partition $(2k+1,2k+1)$. The refined BVLS duality sends $G^\vee.e^\vee$ to the orbit $G.e$ where $e$ has partition $\lambda= (3,2^{2k})$. However, the closure of this orbit is not normal (see, e.g., \cite{Kraft1982OnTG}).
\end{enumerate}
 Recall that we have verified the following isormorphism for various cases in this paper.
\begin{equation} \label{statement}
    H^*(\spr)\cong \bC[\fh\cap \overline{\bO}_e] 
\end{equation}

This version of the Hikita conjecture comes from the identification of the Cartan subquotient $\cC_\nu(\bC[\bO_e])$ with the functions on the scheme theoretic intersection $\bC[\fh\cap \overline{\bO}_e]$, provided $\overline{\bO}_e$ is normal. In general, the Cartan subquotient is defined using a maximal torus of the groups of graded Hamiltonian automorphisms of $\Spec (\bC[\tilde{\bO}_e])$, \cite[Section 9.3]{refinedbvls}. This group is $G$ in most cases, which justifies our consideration of $T$ as a maximal torus of $G$.

For special spherical nilpotent orbits $\bO_e$ in $\fso_{2n+1}$, the statement of Hikita conjecture is slightly different from (\ref{statement}) for the following reason. The nilpotent covers $\tilde{\bO}_e$ belong to the set of shared orbits from \cite{fu2023local}, so they have a larger group of automorphisms. In particular, from \cite[Proposition 2.12]{fu2023local}, we have
\begin{enumerate}
    \item $\bC[\tilde{\bO}_{(3,2^{2k-2},1^{2n-4k+2})}]$ is isomorphic to $\bC[\bO_{(2^{2k},1^{2n-4k+2})}]= \bC[\overline{\bO}_{(2^{2k},1^{2n-4k+2})}]$. Here, $G'.e'= \bO_{(2^{2k},1^{2n-4k+2})}$ is an orbit of $\fg'= \fso_{2n+2}$.
    \item The normalization of $\overline{\bO}_{(3,2^{2k})}$ is isomorphic to $\overline{\bO^I}_{(2^{2k+2})}$. Here, $G'.e'= {\bO^I}_{(2^{2k+2})}$ is an orbit of $\fg'= \fso_{2n+2}$.
\end{enumerate}
The groups of graded Hamiltonian automorphisms of these orbits are $G'$ mentioned above. Hence, the torus that we use to define the Cartan subquotient of $\bC[\bO_e]$ in these cases is a maximal torus $T'$ of $G'$. Write $\fh'\subset \fg'$ for the corresponding Lie algebra. In the statement of the Hikita conjecture for these orbits, the right side of (\ref{statement}) is $\bC[\fh'\cap \overline{G'.e'}]$ instead of $\bC[\fh\cap \overline{\bO}_e]$. Now $\fg'$ is of type D. The partition of $(e^\vee)'$ in the two above cases are $(2n-2k+1, 2k+1)$ and $(2k+2,2k+2)$, respectively. In the previous section, we have proved that $\bC[\fh\cap \overline{G'.e'}]\cong H^*(\cB_{(e^\vee)'})$. Therefore, the Hikita conjecture in this case follows from the fact that the cohomology of $\cB_{(2n-2k,2k)}$ in type C is the same as the cohomology of $\cB_{(2n-2k+1, 2k+1)}$ in type D (\cite[Remark 2]{Stroppel_2018}).


\appendix
\section{Fixed point varieties of Springer fibers} \label{Springerfix}
In this appendix, we conclude the proof of \cref{surjective case}
\begin{lem}\label{free small cases}
    In the following cases, the image of $i^*_{T_\ev\times T_\hv}$ is a free module over $\bC[\ft_\ev,\ft_\hv]= \bC[t_0,h]$.
    \begin{enumerate}
        \item $\fg^\vee$ is of type B, the partition of $e^\vee$ takes the form $(2,2,1)$ or $(1,1,1)$.
        \item $\fg^\vee$ is of type C, the partition of $e^\vee$ takes the form $(2n-2,1,1)$, $(2,2,2)$, or $(4,3,3)$.
        \item $\fg^\vee$ is of type D, the partition of $e^\vee$ takes the form $(2n-3,1,1,1)$ $(3,2,2,1)$, or $(2,2,1,1)$.
    \end{enumerate}
\end{lem}
\begin{proof}
    \begin{enumerate}
    \item We explain the details of the case $\fg^\vee$ of type C. Other cases are similar. Let $V$ be the standard representation of $\fg^\vee$. Write $K$ for the kernel of the map $i^*_{T_\ev\times T_\hv}$. The freeness of the image of $i^*_{T_\ev\times T_\hv}$ is implied by the existence of a set of uniform generators of $K$. In other words, we want to find a set of polynomials $P_1,...,P_m\in \bC[y_1,...,y_n][t_0,h]$ so that the specialization of $P_1,...,P_m$ to $(t_0,h)= (0,0)$ generates $\gr K$, the specialization of $K$ at $(0,0)$. 
    \begin{enumerate}
        \item $\lambda^\vee= (2n-2,1,1)$. The weights of $t= (t_0,h)$ on $V$ are $ t_0, -t_0,(2n-3)h,...,(3-2n)h$. Let $P_{(2n-2,1,1)}$ be the set of $2n$ points $(y_1,...,y_n)\subset \bC^n$ with the following properties.
        \begin{itemize}
            \item $y_i= t$ or $-t$ for some $1\leqslant i\leqslant n$.
            \item After we remove $y_i$, the remaining $n-1$ coordinates are $(2n-3)h,...,h$ in order from left to right.
        \end{itemize}
        By \cite[Section 7]{HMK2024}, $K$ contains the polynomials in $\bC[y_1,...,y_n][t_0,h]$ that vanish on $P_{(2n-2,1,1)}$. Consider $y=(y_1,...,y_n)$ in $P_{(2n-2,1,1)}$. For $1\leqslant i\leqslant n$, if $y_i\neq (2n-1-2i)h$ then $y_j= (2n-1-2j)h$ for $j<i$ and $y_j= (2n+1-2j)h$ for $j>i$. In other words, $K$ contains $(y_i- (2n-1-2i)h)(y_j- (2n-1-2j)h)$ for $j<i$ and $(y_i- (2n-1-2i)h)(y_j- (2n+1-2j)h)$ for $j>i$. Similarly, $K$ contains $(y_i- (2n-1-2i)h)(y_i^{2}-t_0^{2})$. If we specialize $h$ to $0$, then $\gr K$ contains $y_iy_j, y_i^{3}$. These are precisely the generators of $\gr I_x$ mentioned in \cref{bi=1}. Therefore, we have obtained a set of uniform generators for $K$, and the image of $i^*_{T_\ev\times T_\hv}$ is free. 

        \item $\lambda^\vee= (2,2,2)$. The weights of $t=(t_0,h)$ on $V$ are $t_0+h, t_0-h, -t_0+h, -t_0-h, h,-h$. Let $P_{(2,2,2)}$ be the set of $12$ points $(y_1,y_2,y_3)\subset \bC^n$ with the following properties.
        \begin{itemize}
            \item If there is $i\in \{1,2,3\}$ so that $y_i= -h, t_0-h$ or $-t_0-h$ then there exists $1\leqslant j<i$ so that $y_j= y_i+2h$. 
        \end{itemize}
      Similarly to the previous case, we have a set of uniform generators for $K$ that consists of $(y_1-h)(y_2-h)(y_3-h)$, $(y_1-h)(y_1-t_0-h)(y_1+t_0-h)$, $(y_3-h)(y_3^{2}-t_0^{2}+2h(y_1+y_2)-5h^2)$, $(y_2-3h)(y_2^{2}-t_0^{2}+2hy_1-3h^2)+2h(h^2-y_3^{2})$. 
        
        \item $\lambda^\vee= (4,3,3)$. The weights of $t=(t_0,h)$ on $V$ are $t_0+2h, t_0, t_0-2h, -t_0+2h, -t_0, -t_0-2h, 3h, h,-h, -3h$. Let $P_{(4,3,3)}$ be the set of $80$ points $(y_1,...,y_5)\subset \bC^5$ with the following properties.
        \begin{itemize}
            \item There are $1\leqslant i< j\leqslant 5$ so that $y_i= 3h$ and $y_j=  h$. 
            \item Assume that the three remaining coordinates are $y_{i_1}, y_{i_2}, y_{i_3}$, $i_1<i_2<i_3$. Then if $y_{i_l}+ 2h$ belongs to the set of weights of $t$ on $V$, then there is $l'<l$ so that $y_{i_{l'}}= y_{i_l}+  2h$. In particular, $y_{i_1}$ is always $t_0+2h$ or $-t_0-2h$.
        \end{itemize}
        From \cite[Section 7]{HMK2024}, $K$ contains the polynomials in $\bC[y_1,...,y_5][t_0,h]$ that vanish on $P_{(4,3,3)}$. The ideal $\gr K$ is generated by $y_i^{3}$ and $y_iy_jy_ky_l$. On the dual side, the partition of $e$ is $(3,3,1,1,1)$. The elements $y_iy_jy_ky_l$ come from the Pfaffian of size $8$ minors, the elements $y_i^{3}$ come from $A^3=0$ if $A\in \overline{\bO}_e$. We have a corresponding set of uniform generators of $K$ as follows.
        \begin{itemize}
            \item We must have $3h\in \{y_1,...,y_4\}$, so $\prod_{i=1}^{4}(y_i-3h)\in K$. Similarly, $\prod_{i=2}^{5}(y_i-h)\in K$. If $y_5$ is not $h$, then $3h\in \{y_1,y_2,y_3\}$. Hence $(y_5-h)(y_1-3h)(y_2-3h)(y_3-3h)\in K$. Similarly, we have $(y_5-h)(y_4-h)(y_1-3h)(y_2-3h)$ and $(y_5-h)(y_4-h)(y_3-h)(y_1-3h)$ in $K$. The leading terms of these elements give the elements $y_iy_jy_ky_l$ in $\gr K$.
            \item We have $y_1$ is $3h$, $t_0+2h$, or $t_0-2h$. Hence $(y_1- 3h)(y_1-t_0-2h)(y_1+t_0-2h)\in K$, this element has the leading term $y_1^{3}\in \gr K$. Similarly, we have elements with leading terms $y_2^{3},...,y_5^{3}$ in $K$. 
        \end{itemize}
        \end{enumerate}
         \item Consider $\fg^\vee$ of type B and $\lambda^\vee= (2,2,1)$. The weights of $t$ on $V$ are $t_0+h, t_0-h, -t_0+h, -t_0-h,0$. The set $P_{(2,2,1)}$ consists of $4$ points $(t_0+h, t_0-h)$, $(-t_0+h, -t_0-h)$, $(t_0+h, -t_0+h)$ and $(-t_0+h, t_0+h)$ in $\bC^2$. The ideal $K$ is generated by $(y_1-t_0-h)(y_1+t_0-h)$, $(y_2+y_1-2h)(y_2-y_1+2h)$. The leading terms of these elements are $y_1^2$ and $y_2^2$. On the dual side, the partition of $e$ is $(2,2)$ and $\bC[\overline{\bO}_e\cap \fh]= \bC[y_1,y_2]/(y_1^{2}, y_2^{2})$, see \cref{flat sp}.
         \item Consider $\fg^\vee$ of type D.
         \begin{enumerate}
             \item $\lambda^\vee= (2n-3,1,1,1)$. This case is similar to the case $\lambda^\vee= (2n-2,1,1)$ in type C.
             \item $\lambda^\vee= (2,2,1,1)$. The weights of $t=(t_0,h)$ on $V$ are $t_0+h, t_0-h, -t_0+h, -t_0-h,0,0$. The set $P_{(2,2,1,1)}$ consists of $12$ points $(y_1,y_2,y_3)$ with the following property.
             \begin{itemize}
                 \item One of the coordinates is $0$. The two nonzero coordinates are $(t_0+h, t_0-h)$, $(-t_0+h, -t_0-h)$, $(t_0+h, -t_0+h)$, or $(-t_0+h, t_0+h)$.
             \end{itemize}
             This case is similar to the case $\lambda^\vee= (2,2,2)$ in type C. The ideal $\gr K$ is generated by $y_i^{3}$ and $y_1y_2y_3$ as in \cref{b1=2}.
             \item $\lambda^\vee= (3,2,2,1)$. The weights of $t=(t_0,h)$ on $V$ are $t_0+h, t_0-h, -t_0+h, -t_0-h,2h,0,0,2h$. The set $P_{(2,2,1,1)}$ consists of $24$ points $(y_1,y_2,y_3,y_4)$ with the following property.
             \begin{itemize}
                \item There is $1\leqslant i< j\leqslant 4$ so that $y_i= 2h$ and $y_j= 0$.
                 \item The two other coordinates are $(t_0+h, t_0-h)$, $(-t_0+h, -t_0-h)$, $(t_0+h, -t_0+h)$, or $(-t_0+h, t_0+h)$.
             \end{itemize}
             The ideal $\gr K$ is generated by $y_i^{3}$ and $y_iy_jy_k$ as in \cref{b1=2}.
         \end{enumerate}

    \end{enumerate}

\end{proof}

\begin{lem}\label{fixed pt}
    Consider $\ev$ in \cref{free small cases}. For $t\in \ft_\ev\oplus \ft_\hv= \bC^2$ and $t\neq (0,0)$, the pullback map $i^*_{t}: H^*((\bv)^{t})\rightarrow H^*((\spr)^t)$ is surjective. Here, $X^t$ denotes the fixed-point subvariety of $X$ under the action of the vector field induced by $t$. 
\end{lem}
\begin{proof}
    Write $t$ as $(t_0, h)\in \ft_\ev\oplus \ft_\hv$. The fixed point locus depends only on the ratio of $t_0$ and $h$. Hence, we only need to consider $h=0$ or $h=1$. For $t_0=0$ and $h\neq 0$, the map $i^*_{t}$ is surjective by \cite[Lemma 10.9]{HMK2024}. Consider $h=1$ and $t_0\neq 0$. View $t$ as an element of $\fg$, then $t$ acts on the standard representation of $\fg$. The points of $(\bv)^{t}$ and $(\spr)^t$ are isotropic flags of $V$ that are stable under $t$. Therefore, if the weights of $t$ on $V$ are pairwise distinct, both $(\bv)^{t}$ and $(\spr)^t$ are discrete sets of points. Consequently, the map is $i^*_{t}$ surjective. In the remaining cases, we describe the fixed point varieties $X^t$ for $t=(t_0, 1)$ and $X= \bv, \spr$. The surjectivity of $i_t^{*}$ will naturally follow from these descriptions.  
    
      
    \begin{enumerate}
        \item Case $\fg^\vee$ of type B and $\lambda^\vee= (2,2,1)$. The weight of $t$ on $V$ are $0,t_0+1,t_0-1,-t_0+1,-t_0-1$. The cases we need to consider are $t_0=\pm 1$. The weights on $V$ are $0,2,0,-2,0$. The components of $(\bv)^{t}$ are isomorphic to the flag of $\fso(V_0)$, hence isomorphic to $\bP^1$. The components of $(\spr)^t$ are either $\bP^1$ or a single point.
        \item Case $\fg^\vee$ of type C and $\lambda^\vee= (2,2,2)$. The weights of $t$ on $V$ are $1,-1,t_0+1,t_0-1,-t_0+1,-t_0-1$. The cases we need to consider are $t_0=\pm 1$ and $t_0=\pm 2$.
        \begin{itemize}
            \item Consider $t_0=\pm 2$, the weights are $1,-1,3,1,-1,-3$. The components of $(\bv)^{t}$ are isomorphic to the flag of $\fsp(V_1\oplus V_{-1})$, hence isomorphic to $(\bP^3,\bP^1)$. The components of $(\spr)^t$ are either $\bP^1$ or a single point.
            \item  Consider $t_0=\pm 1$, the weights are $1,-1,2,0,0,-2$. The components of $(\bv)^{t}$ are isomorphic to the flag of $\fsp(V_0)$, hence isomorphic to $\bP^1$. The components of $(\spr)^t$ are either $\bP^1$ or a single point.
        \end{itemize}
        \item Case $\fg^\vee$ of type C and $\lambda^\vee= (2n-2,1,1)$. The weights of $t$ on $V$ are $t_0,-t_0, 2n-3,...,3-2n$. The cases we need to consider are $t_0\in \{2n-3,...,3-2n\}$, $t\neq 0$. The components of $(\bv)^{t}$ are isomorphic to the flag of $\fsp(V_{t_0}\oplus V_{-t_0})$, hence isomorphic to $(\bP^3,\bP^1)$. The components of $(\spr)^t$ are either $\bP^1$ or a single point.
        \item Case $\fg^\vee$ of type D and $\lambda^\vee= (2n-3,1,1,1)$. The weights of $t$ on $V$ are $t_0,-t_0, 2n-2,...,0,0,...,2-2n$. The cases we need to consider are $t_0\in \{2n-2,...,2-2n\}$, $t_0\neq 0$. The components of $(\bv)^{t}$ are isomorphic to the flag of $\fso(V_{t_0}\oplus V_{-t_0})$, hence isomorphic to $\bP^1\times \bP^1$. The components of $(\spr)^t$ are either $\bP^1$ or a single point.
        \item Case $\fg^\vee$ of type D and $\lambda^\vee= (2,2,1,1)$. The weights of $t$ on $V$ are $0,0,t_0+1,t_0-1,-t+1,-t-1$. The cases we need to consider are $t=\pm 1,\pm 2$. The result is similar to the case $\fg^\vee$ of type B and $\lambda^\vee= (2,2,1)$.
        \item Consider $\fg^\vee$ of type D, $\lambda^\vee= (3,2,2,1)$. The weights of $t$ on $V$ are $2,0,0,-2,t_0+1,t_0-1,-t_0+1,-t_0-1$. The cases we need to consider are $t_0=\pm 1, \pm 3$. Elements of $(\spr)^t$ can be written as isotropic flags that are represented by eigenvectors $(u_1,u_2,u_3,u_4)$ of $t$ in $V$. The components of $(\spr)^t$ are parameterized by the weights of $u_i$. Write $X_\alpha$ for the components that correspond to the weight tuple $\alpha$.
    \begin{itemize}
            \item Consider $t_0=\pm 1$, the weights are $2,2,0,0,0,0,-2,-2$. The components of $(\bv)^{t}$ are isomorphic to the flag of $\fgl_2\times \fso(V_0)$, hence isomorphic to $\bP^1\times \bP^1\times\bP^1$. The weight tuples $\alpha$ so that $X_\alpha$ is not empty are $(2,2,0,0)$, $(2,0,2,0)$, $(0,2,2,0)$, $(0,2,0,2)$, $(2,0,0,-2)$, $(0,2,0,-2)$ and $(0,-2,2,0)$. The corresponding varieties $X_\alpha$ are as follows. 
            \begin{enumerate}
                \item $X_{(2,2,0,0)}\cong \bP^1\times \bP^1\times \bP^1$.
                \item $X_{(2,0,2,0)}$ is the blow-up of $\bP^1\times \bP^1$ at a point.
                \item $X_{(0,2,0,2)}$ is a single point.
                \item $X_{(2,0,0,-2)}\cong \bP^1$.
                \item $X_{(0,2,0,-2)}$ is a single point.
                \item $X_{(0,-2,2,0)}$ is a single point.
            \end{enumerate}
            
            \item Consider $t_0=\pm 3$, the weights are $4,2,2,0,0,-2,-2,-4$. The components of $(\bv)^{t}$ are isomorphic to the flag of $\fgl_1\times \fgl_2\times \fso(V_0)$, hence isomorphic to $\bP^1$. The components of $(\spr)^t$ are either $\bP^1$ or a single point.
    \end{itemize}
    \item Consider $\fg^\vee$ of type C, $\lambda^\vee= (4,3,3)$. The weights of $t$ on $V$ are $t_0+2, t_0, t_0-2, -t_0+2, -t_0, -t_0-2, 3,1,-1,-3$. The cases we need to consider are $t_0= \pm 1,3$. We use the notation $X_\alpha$ as in the previous case.
    \begin{itemize}
            \item Consider $t_0=\pm 1$, the weights are $3,3,1,1,1,-1,-1,-1,-3,-3$. The components of $(\bv)^{t}$ are isomorphic to the flag of $\fgl_2\times \fgl_3$, hence isomorphic to $\bP^1\times (\bP^1,\bP^2)$. If the tuple $\alpha$ contains the weight $-3$, then $X_\alpha$ is either a point or $\bP^1$. We consider cases where $\alpha$ does not contain $-3$. If $\alpha$ does not contain the weight $-1$, then $X_\alpha$ is a tower of projective bundles, see \cite{hoang2024geometryfixedpointsloci}. In the remaining cases, we have
            
            \begin{enumerate}
                \item $X_{(3,3,1,1,-1)}\cong \bP^1\times (\bP^1,\bP^1)$.
                \item $X_{(3,3,1,-1,1)}\cong \bP^1\times \bP^1$.
                \item $X_{(3,1,3,1,-1)}$ is the blow-up of $\bP^1\times \bP^1$ at a single point. 
                \item $X_{(3,1,1,-1,3)}\cong X_{(3,1,1,3,-1)}\cong \bP^1$. 
                \item $X_{(3,1,-1,3,1)}$ is the union of two $\bP^1$ at a point. This is an example that $X_\alpha$ is not smooth. 
            \end{enumerate}
            \item Consider $t_0=\pm 3$, the weights are $3,1,-1,-3,5,3,1,-1,-3,-5$. The components of $(\bv)^{t}$ are isomorphic to the flag of $\fgl_2\times \fgl_2$, hence isomorphic to $\bP^1\times \bP^1$. The components of $(\spr)^t$ are $\bP^1\times \bP^1$ ,$\bP^1$ or a single point.
    \end{itemize}
    \end{enumerate} 
\end{proof}

\printbibliography
\end{document}